\documentclass[10pt,reqno]{amsart}
\usepackage{amsmath,amssymb,latexsym,esint,cite,mathrsfs}
\usepackage{verbatim,wasysym,cite,relsize,color}
\usepackage[left=2.5cm,right=2.5cm,top=2.5cm,bottom=2.5cm]{geometry}
\usepackage[latin1]{inputenc}
\usepackage{microtype}
\usepackage{color,enumitem,graphicx}
\usepackage[colorlinks=true,urlcolor=blue, citecolor=red,linkcolor=blue,
linktocpage,pdfpagelabels, bookmarksnumbered,bookmarksopen]{hyperref}
\usepackage[hyperpageref]{backref}
\usepackage[english]{babel}
\usepackage{tikz}
\usepackage[font=small,skip=5pt]{caption}
\numberwithin{equation}{section}
\newtheorem{theorem}{Theorem}[section]
\newtheorem{lemma}[theorem]{Lemma}
\newtheorem{proposition}[theorem]{Proposition}

\theoremstyle{definition}
\newtheorem{definition}[theorem]{Definition}
\newtheorem{remark}[theorem]{Remark}

\newcommand{\N}{\mathbb N}

\newcommand{\R}{\mathbb R}
\newcommand{\C}{\mathbb C}
\newcommand{\Sb}{\mathbb S}

\newcommand{\varep}{\varepsilon}

\title[NLS Attractive Inverse-power Potentials]
{On nonlinear Schr\"odinger equations with attractive inverse-power potentials}
\author[V. D. Dinh]{Van Duong Dinh}
\address[V. D. Dinh]{Laboratoire Paul Painlev\'e UMR 8524, Universit\'e de Lille CNRS, 59655 Villeneuve d'Ascq Cedex, France
and 
Department of Mathematics, HCMC University of Pedagogy, 280 An Duong Vuong, Ho Chi Minh, Vietnam}
\email{contact@duongdinh.com}


\begin{document}
	
	\begin{abstract}
	We study the Cauchy problem for nonlinear Schr\"odinger equations with attractive inverse-power potentials. By using variational arguments, we first determine a sharp threshold of global well-posedness and blow-up for the equation in the mass-supercritical case. We next study the existence and orbital stability of standing waves for the problem in the mass-subcritical and mass-critical cases. In the mass-critical case, we give a detailed description of the blow-up behavior of standing waves when the mass tends to a critical value. 
	\end{abstract}

	\maketitle

	\section{Introduction and main results}
	\label{S1}
	We consider the Cauchy problem for nonlinear Schr\"odinger equations with attractive inverse-power potentials
	\begin{equation} \label{NLS-att}
		\left\{ 
		\begin{array}{rcl}
			i\partial_t u + \Delta u + |x|^{-\sigma} u &=& \pm |u|^\alpha u, \quad (t,x) \in \R \times \R^d, \\
			u(0)&=& u_0,
		\end{array}
		\right.
	\end{equation}
	where $u: \mathbb{R} \times \mathbb{R}^d \rightarrow \mathbb{C}$, $u_0: \mathbb{R}^d \rightarrow \mathbb{C}$, $0<\sigma<\min\{2,d\}$ and $\alpha>0$. The plus (resp. minus) sign in front of the nonlinearity corresponds to the defocusing (resp. focusing) case. 
	
	The Schr\"odinger equations with inverse-power potentials have attracted much attention recently. In the case of inverse-square potential $\sigma=2$, one has the following results: see Burq-Planchon-Stalker-Tahvildar-Zadeh \cite{BPST} for Strichartz estimates; Okazawa-Suzuki-Yokota \cite{OSY} for the local and global well-posedness of the Cauchy problem; Zhang-Zheng \cite{ZZ} for the energy scattering in the defocusing case; Killip-Miao-Visan-Zhang-Zheng \cite{KMVZZ-sob} for Sobolev spaces adapted to the Schr\"odinger operator with inverse-square potentials; Killip-Murphy-Visan-Zheng \cite{KMVZ}, Zheng \cite{Zheng}, Lu-Miao-Murphy \cite{LMM} and Dinh \cite{Dinh-inv} for the global existence, blow-up and energy scattering in the focusing case; Killip-Miao-Visan-Zhang-Zheng \cite{KMVZZ} for the global existence and scattering in the energy-critical case; Csobo-Genoud \cite{CG} for the classification of minimial mass blow-up solutions and Bensouilah-Dinh-Zhu \cite{BDZ} for the stability and instability of standing waves. In the case of Coulomb potential $\sigma=1$, one has results of Benguria-Jeanneret \cite{BJ} for the existence and uniqueness of positive solutions of semilinear elliptic equations; Chadam-Glassey \cite{ChG}, Hayashi-Ozawa \cite{HO} and Lions \cite{Lions} for the global existence and time decay of global solutions for the Hartree equations and Miao-Zhang-Zheng \cite{MZZ} for the global existence, blow-up and energy scattering. In the case of slowly decaying potentials $0<\sigma<2$, we refer to Mizutani \cite{Mizutani} for Strichartz estimates; Fukaya-Ohta \cite{FO} for the strong instability of standing waves; Guo-Wang-Yao \cite{GWY} for the blow-up and energy scattering of the focusing 3D cubic NLS and Li-Zhao \cite{Li-Zhao} for the orbital stability of standing waves.
	
	This paper is a continuation of \cite{Dinh-rep} where nonlinear Schr\"odinger equations with repulsive (i.e. the minus sign in front of $|x|^{-\sigma}u$) inverse-power potentials in the energy space were considered. The local well-posedness (LWP) for \eqref{NLS-att} in the energy space $H^1$ was studied in \cite{Dinh-rep}. More precisely, the author showed that \eqref{NLS-att} is locally well-posed in $H^1$ for both energy-subcritical and energy-critical cases. Moreover, the time of existence depends only on the $H^1$-norm of initial data. In the energy-subcritical case, this LWP coincides with the usual local theory. In the energy-critical case, this LWP is stronger than the usual one in which the time of existence depends not only on the $H^1$-norm of initial data but also on its profile. The proof is based on the perturbation argument of Zhang in \cite{Zhang} by using Strichartz estimates in Lorentz spaces and viewing the potential as a sub-critical nonlinear term. A direct consequence of this local theory is the following global well-posedness result for \eqref{NLS-att}.
	\begin{theorem} [Global existence \cite{Dinh-rep}] \label{theo-gwp-1}
		Let $u_0 \in H^1$. Suppose that
		\begin{itemize}
			\item in the defocusing case:
			\begin{itemize}
				\item (Energy-subcritical case) $0<\sigma<\min\{2,d\}$ and $0<\alpha<\frac{4}{d-2}$ if $d\geq 3$ ($0<\alpha<\infty$ if $d=1,2$);
				\item (Energy-critical case) $0<\sigma<2$ if $d\geq 4$  ($0<\sigma<\frac{3}{2}$ if $d=3$) and $\alpha=\frac{4}{d-2}$;
			\end{itemize}
			\item in the focusing case:
			\begin{itemize}
				\item (Mass-subcritical case) $0<\sigma<\min\{2,d\}$ and $0<\alpha<\frac{4}{d}$;
				\item (Mass-critical case) $0<\sigma<\min\{2,d\}$, $\alpha=\frac{4}{d}$ and $\|u_0\|_{L^2} <\|Q\|_{L^2}$, where $Q$ is the unique (up to symmetries) positive radial solution to the elliptic equation
				\begin{align} \label{ell-equ-mas-cri}
				\Delta Q- Q + |Q|^{\frac{4}{d}} Q =0.
				\end{align}
			\end{itemize}
		\end{itemize}
		Then there exists a unique global solution to \eqref{NLS-att}. Moreover, the global solution $u$ satisfies for $0<\sigma<2$ if $d\geq 4$ ($0<\sigma<\frac{3}{2}$ if $d=3$) and any compact interval $J\subset \R$,
		\[
		\sup_{(p,q)\in S} \|u\|_{L^p(J, W^{1,q})} \leq C(\|u_0\|_{H^1}, |J|),
		\]
		where $(p,q) \in S$ means that $(p,q)$ is a Schr\"odinger admissible pair.
	\end{theorem}
	
	\subsection{Sharp threshold for global existence and blow-up}
	The first part of this paper is devoted to the global well-posedness for the focusing problem \eqref{NLS-att} in the mass-supercritical case. Before stating our results, let us introduce some notations. By a standing wave, we mean a solution to the focusing problem \eqref{NLS-att} of the form $e^{i\omega t} \phi(x)$, where $\omega \in \R$ is a frequency and $\phi \in H^1$ is a nontrivial solution to the elliptic equation
	\begin{align} \label{ell-equ-att}
	-\Delta \phi + \omega \phi -  |x|^{-\sigma} \phi - |\phi|^\alpha \phi =0. 
	\end{align}
	It is well-known (see e.g. \cite[Theorem 8.38]{GT}) that the minimizing problem
	\[
	\mu_1:= \inf\left\{ \|\nabla v\|^2_{L^2} - G(v) \ : \ v \in H^1, \|v\|_{L^2}=1 \right\}
	\]
	is attained by a positive function $\Phi$, where 
	\begin{align} \label{def-G}
	G(v):= \int |x|^{-\sigma} |v(x)|^2 dx.
	\end{align}
	Moreover, $\mu_1$ is the simple first eigenvalue corresponding to the eigenfunction $\Phi$ of the operator $-\Delta - |x|^{-\sigma}$. We also have from the Virial Theorem (see e.g. \cite[Theorem 6.2.8]{deOliveira}) that $\mu_1$ is negative. By the definition of $\mu_1$, we see that
	\begin{align} \label{eig-est}
	\mu_1 \|v\|^2_{L^2} \leq \|\nabla v\|^2_{L^2} - G(v)
	\end{align}
	for any $v \in H^1$. It is worth noticing that if $\omega \leq -\mu_1$, then the equation \eqref{ell-equ-att} does not admit positive solutions. In fact, suppose that $\phi$ is a positive solution of \eqref{ell-equ-att}. By multiplying both sides of \eqref{ell-equ-att} with $\Phi$ and integrating by parts, we get
	\[
	(\omega + \mu_1) \int \phi(x) \Phi(x) dx = \int \phi^{\alpha+1}(x) \Phi(x) >0.
	\]
	In the case $\omega>-\mu_1$, there exists at least one solution to \eqref{ell-equ-att} which is spherically symmetric and positive. To see this, we define the following action functional
	\begin{align*}
	S_\omega(v):= E(v) + \frac{\omega}{2} M(v) = \frac{1}{2} \|\nabla v\|^2_{L^2} -\frac{1}{2} G(v) + \frac{\omega}{2} \|v\|^2_{L^2} - \frac{1}{\alpha+2} \|v\|^{\alpha+2}_{L^{\alpha+2}},
	\end{align*}
	and the corresponding Nehari functional
	\begin{align*}
	K_\omega(v) := \left.\partial_\lambda S_\omega(\lambda v) \right|_{\lambda=1} = \|\nabla v\|^2_{L^2} - G(v) + \omega \|v\|^2_{L^2} - \|v\|^{\alpha+2}_{L^{\alpha+2}}.
	\end{align*}
	Note that the elliptic equation \eqref{ell-equ-att} can be written as $S'_\omega(\phi) =0$. Consider the minimizing problem 
	\begin{align} \label{d-ome}
	d(\omega) := \inf \left\{S_\omega(v) \ : \ v \in H^1 \backslash \{0\}, K_\omega(v) =0 \right\},
	\end{align}
	and define the set of all minimizers for \eqref{d-ome} by
	\[
	\mathcal{M}_\omega:= \left\{ v \in H^1\backslash \{0\} \ : \ K_\omega(v) =0, S_\omega(v) = d(\omega) \right\}.
	\]
	
	\begin{proposition}  \label{prop-exi-min}
		Let $d\geq 1$, $0<\sigma <\min\{2,d\}$ and $0<\alpha<\frac{4}{d-2}$ if $d\geq 3$ ($0<\alpha<\infty$ if $d=1,2$). If $\omega>-\mu_1$, then $d(\omega)>0$ and $d(\omega)$ is attained by a function which is a solution to the elliptic equation \eqref{ell-equ-att}. Moreover, every minimizer for $d(\omega)$ is of the form $e^{i\theta} \phi(x)$, where $\phi$ is a positive radially symmetric function.
	\end{proposition}
	
	Notice that $\phi$ solves the ordinary differential equation
	\begin{align} \label{ordinary-equation}
	\phi''(r) + \frac{d-1}{r} \phi'(r) + (r^{-\sigma} -\omega) \phi(r) + \phi^{\alpha+1}(r) = 0, \quad r\in (0,+\infty).
	\end{align}
	Using the general results of Shioji-Watanabe \cite{SW}, we have (see Appendix) that for any $\omega>-\mu_1$, $0<\sigma<1$, $d\geq 3$ and $0<\alpha<\frac{4}{d-2}$, there exists a unique positive solution to \eqref{ordinary-equation}. A same argument has been used in \cite{AD} to show the uniqueness of positive ground states for the inhomogeneous Gross-Pitaevskii equation.
	
	We now denote the set of nontrivial solutions of \eqref{ell-equ-att} by
	\[
	\mathcal{A}_\omega:= \left\{ v \in H^1 \backslash \{0\} \ : \ S'_\omega(v) =0\right\}.
	\]
	\begin{definition}
		A function $\phi \in \mathcal{A}_\omega$ is called a ground state for \eqref{ell-equ-att} if it minimizes $S_\omega$ over the set $\mathcal{A}_\omega$. The set of ground states for \eqref{ell-equ-att} is denoted by $\mathcal{G}_\omega$. In particular,
		\[
		\mathcal{G}_\omega = \left\{ \phi \in \mathcal{A}_\omega \ : \ S_\omega(\phi) \leq S_\omega(v), \ \forall v \in \mathcal{A}_\omega \right\}.
		\]
	\end{definition}

	\begin{proposition} [Existence of ground states \cite{FO}] \label{prop-exi-gro-sta}
		Let $d\geq 1$, $0<\sigma <\min\{2,d\}$ and $0<\alpha<\frac{4}{d-2}$ if $d\geq 3$	($0<\alpha<\infty$ if $d=1,2$). If $\omega>-\mu_1$, then the set $\mathcal{G}_\omega$ is not empty, and it is characterized by
		\[
		\mathcal{G}_\omega = \left\{ v \in H^1 \backslash \{0\} \ : \ S_\omega(v) = d(\omega), \ K_\omega(v) =0 \right\}.
		\]		
	\end{proposition}
	
	We now denote the functional
	\[
	Q(v):= \|\nabla v\|^2_{L^2} - \frac{\sigma}{2} G(v) - \frac{\beta}{\alpha+2} \|v\|^{\alpha+2}_{L^{\alpha+2}},
	\]
	where
	\begin{align} \label{def-bet}
	\beta:= \frac{d\alpha}{2}.
	\end{align}
	The functional $Q$ comes from the virial action
	\begin{align} \label{virial-action}
	\frac{d^2}{dt^2} \|x u(t)\|^2_{L^2} = 8 Q(u(t)),
	\end{align}
	where $u$ is the solution to the focusing problem \eqref{NLS-att}. Let $\phi \in \mathcal{G}_\omega$. We define the following sets
	\begin{align} \label{inv-set}
	\begin{aligned}
	\mathcal{K}^-_\omega &:= \left\{ v\in H^1 \backslash \{0\} \ : \ \|v\|_{L^2} \leq \|\phi\|_{L^2}, S_\omega(v) < S_\omega(\phi), K_\omega(v)<0, Q(v) <0 \right\}, \\
	\mathcal{K}^+_\omega &:= \left\{ v\in H^1 \backslash \{0\} \ : \ \|v\|_{L^2} \leq \|\phi\|_{L^2}, S_\omega(v) < S_\omega(\phi), K_\omega(v)<0, Q(v) >0 \right\}. \\
	\end{aligned}
	\end{align}
	We will see in Remark $\ref{rem-intersection}$ and Lemma $\ref{lem-large-omega}$ that for $\omega$ large enough,
	\[
	\left\{v\in H^1 \backslash \{0\} \ : \ \|v\|_{L^2} \leq \|\phi\|_{L^2}, S_\omega(v) < S_\omega(\phi), K_\omega(v)<0, Q(v)=0 \right\} = \emptyset,
	\]
	hence
	\[
	\mathcal{K}^-_\omega \cup \mathcal{K}^+_\omega = \left\{v\in H^1 \backslash \{0\} \ : \ \|v\|_{L^2} \leq \|\phi\|_{L^2}, S_\omega(v) < S_\omega(\phi), K_\omega(v)<0 \right\}.
	\]
	We are now able to state our first result concerning the sharp threshold of global existence and blow-up for the focusing problem \eqref{NLS-att} in the mass-supercritical and energy-subcritical case.
	\begin{theorem} \label{theo-sha-threshold}
		Let $d\geq 1$, $0<\sigma<\min\{2,d\}$ and $\frac{4}{d}<\alpha<\frac{4}{d-2}$ if $d\geq 3$ ($\frac{4}{d}<\alpha<\infty$ if $d=1,2$). Then there exists $\omega_0>-\mu_1$ such that for any $\omega \geq \omega_0$ and $\phi \in \mathcal{G}_\omega$, the following properties hold:
		\begin{itemize}
			\item If $u_0 \in \mathcal{K}^-_\omega$ and $|x| u_0 \in L^2$, then the corresponding solution to \eqref{NLS-att} blows up in finite time.
			\item If $u_0 \in \mathcal{K}^+_\omega$, then the corresponding solution to \eqref{NLS-att} exists globally in time.
		\end{itemize}
	\end{theorem}
	
	The proof of finite time blow-up given in Theorem $\ref{theo-sha-threshold}$ is based on the variational argument of  \cite{FO}. The key point (see Proposition $\ref{prop-key-est}$) is to show for $\omega>-\mu_1$ large enough and $v \in H^1\backslash \{0\}$ satisfying 
	\[
	\|v\|_{L^2} \leq \|\phi\|_{L^2}, \quad K_\omega(v) \leq 0, \quad Q(v) \leq 0,
	\]
	it holds that
	\begin{align} \label{estimate-Q}
	Q(v) \leq 2(S_\omega(v) - S_\omega(\phi)).
	\end{align}
	The finite time blow-up then follows from \eqref{virial-action}, \eqref{estimate-Q} and a classical convexity argument of Glassey \cite{Glassey}. We refer the reader to Section $\ref{S3}$ for more details.
	
	\subsection{Existence and stability of standing waves}
	The second part of this paper is devoted to the existence and stability of standing waves for the focusing problem \eqref{NLS-att} in the mass-subcritical and mass-critical cases. Given $a>0$, we consider the minimizing problem
	\begin{align} \label{min-prob-a}
	I(a):= \inf \left\{ E(v) \ : \ v \in H^1, \|v\|^2_{L^2}=a\right\},
	\end{align}
	where
	\[
	E(v):= \frac{1}{2} \|\nabla v\|^2_{L^2} - \frac{1}{2} G(v) - \frac{1}{\alpha+2} \|v\|^{\alpha+2}_{L^{\alpha+2}}.
	\]
	We denote the set of minimizers for $I(a)$ by
	\[
	\mathcal{N}(a) := \left\{v \in H^1 \ : \ E(v) = I(a), \|v\|^2_{L^2}=a \right\}.
	\]
	By the Lagrange multiplier theorem, for each $v \in \mathcal{N}(a)$, there exists $\omega \in \R$ such that \eqref{ell-equ-att} holds with $v$ in place of $\phi$. In this case, $e^{i\omega t} v(x)$ is a solution to \eqref{NLS-att} with initial data $v$. One usually calls $e^{i\omega t} v$ the orbit of $v$. Moreover, if $v \in \mathcal{N}(a)$, i.e. $v$ is a minimizer for $I(a)$, then $e^{i\omega t} v$ is also a minimizer for $I(a)$ or $e^{i\omega t} v \in \mathcal{N}(a)$. We are also interested in the orbital stability for $\mathcal{N}(a)$ under the flow of the focusing problem \eqref{NLS-att}. 
	\begin{definition} 
		The set $\mathcal{N}(a)$ is called orbitally stable under the flow of the focusing problem \eqref{NLS-att} if for every $\varep>0$, there exists $\delta>0$ such that for any initial data $u_0 \in H^1$ satisfying 
		\[
		\inf_{v \in \mathcal{N}(a)} \|u_0 -v\|_{H^1} <\delta,
		\]
		the corresponding solution $u$ to \eqref{NLS-att} satisfies
		\[
		\inf_{v\in \mathcal{N}(a)} \|u(t)-v\|_{H^1} <\varep
		\]
		for all $t \in \R$. 
	\end{definition} 

	Note that the above definition of orbital stability implicitly requires that \eqref{NLS-att} has a unique global solution at least for initial data $u_0$ sufficiently close to $\mathcal{N}(a)$. 
	
	\begin{remark} 
		In the case of no potential and focusing mass-subcritical nonlinearity (i.e. $d\alpha<4$), by using the scaling technique, we can show that each $v \in \mathcal{N}(a)$ is actually a ground state for 
		\begin{align} \label{ell-equ}
		-\Delta \psi +\omega \psi - |\psi|^\alpha \psi =0,
		\end{align}
		where $\omega$ is the Lagrange multiplier, that is, $v$ minimizes the action functional $S_\omega$ over all solutions of \eqref{ell-equ}. In fact, we will show that
		\begin{align} \label{gro-sta}
		S_\omega(v) \leq S_\omega(\psi)
		\end{align}
		for any solution $\psi$ of \eqref{ell-equ}, where
		\[
		S_\omega(v)= \frac{1}{2} \|\nabla v\|^2_{L^2} + \frac{\omega}{2} \|v\|^2_{L^2} -\frac{1}{\alpha+2} \|v\|^{\alpha+2}_{L^{\alpha+2}}.
		\] 
		Assume by contradiction that there exists a solution $\psi$ to \eqref{ell-equ} such that $S_\omega(\psi) <S_\omega(v)$. Since $\psi$ is a solution of \eqref{ell-equ}, we have the following Pohozaev identities
		\[
		\|\nabla \psi\|^2_{L^2} + \omega \|\psi\|^2_{L^2} - \|\psi\|^{\alpha+2}_{L^{\alpha+2}} =0, \quad \frac{2-d}{2} \|\nabla \psi\|^2_{L^2} -\frac{d\omega}{2} \|\psi\|^2_{L^2} +\frac{d}{\alpha+2} \|\psi\|^{\alpha+2}_{L^{\alpha+2}} =0.
		\]
		Of course, similar identities hold for $v$ as well. From these identities, we infer that
		\[
		\|\psi\|^{\alpha+2}_{L^{\alpha+2}} = \frac{2(\alpha+2)}{\alpha}S_\omega(\psi), \quad \|\nabla \psi\|^2_{L^2} = dS_\omega(\psi), \quad \omega \|\psi\|^2_{L^2} = \frac{4-(d-2)\alpha}{\alpha} S_\omega(\psi),
		\]
		and $E(\psi) = \frac{(d\alpha-4)\omega}{2(4-(d-2)\alpha)} \|\psi\|^2_{L^2}$. Now set
		\[
		\lambda:= \left(\frac{a}{\|\psi\|^2_{L^2}} \right)^{\frac{\alpha}{4-d\alpha}}
		\]
		and define
		\[
		\psi_\lambda(x):= \lambda^{\frac{2}{\alpha}} \psi(\lambda x).
		\]
		We see that
		\[
		\|\psi_\lambda\|^2_{L^2} = \lambda^{\frac{4-d\alpha}{\alpha}} \|\psi\|^2_{L^2} = a.
		\]
		Since $v \in \mathcal{N}(a)$, we have
		\[
		\frac{(d\alpha-4)\omega}{2(4-(d-2)\alpha)} a = E(v) \leq E(\psi_\lambda) = \lambda^{\frac{4-(d-2)\alpha}{\alpha}} E(\psi) = \lambda^{\frac{4-(d-2)\alpha}{\alpha}} \frac{(d\alpha-4)\omega}{2(4-(d-2)\alpha)} \|\psi\|^2_{L^2}.
		\]
		Since $d\alpha<4$, it follows that $a \geq \lambda^{\frac{4-(d-2)\alpha}{\alpha}} \|\psi\|^2_{L^2}$, hence 
		\[
		\lambda^{\frac{4-(d-2)\alpha}{\alpha}} \leq \frac{a}{\|\psi\|^2_{L^2}} = \lambda^{\frac{4-d\alpha}{\alpha}} \text{ or } \lambda \leq 1.
		\]
		On the other hand,
		\[
		\frac{\alpha}{4-(d-2)\alpha} \omega \|\psi\|^2_{L^2} = S_\omega(\psi) <S_\omega(v) = \frac{\alpha}{4-(d-2)\alpha} \omega a,
		\]
		hence $\|\psi\|^2_{L^2} <a$ or $\lambda>1$ which is a contradiction. Thus, one gets \eqref{gro-sta} and the claim follows. In the presence of inverse-power potential, there is no scaling invariance for \eqref{NLS-att}, so it is not clear whether or not each $v\in \mathcal{N}(a)$ is a ground state for \eqref{ell-equ-att}.
	\end{remark}
	
	Recently, Li-Zhao \cite{Li-Zhao} studied the existence of standing waves and the orbital stability for $\mathcal{N}(a)$ in the mass-subcritical and mass-critical cases. Their proof is based on the concentration-compactness principle of P. L. Lions \cite{Lions1}. Our purpose here is to give a direct simple proof for the result of \cite{Li-Zhao}.
	
	In the mass-subcritical case, i.e. $\alpha<\frac{4}{d}$, the energy functional is bounded from below on 
	\[
	S(a):=\{ v \in H^1 \ : \ \|v\|^2_{L^2} =a\}.
	\]
	Thus for every $a>0$, we can find the global minimizer of the energy functional on $S(a)$. More precisely, we have the following result.
	\begin{theorem} \label{theo-exi-min-mas-sub}
		Let $d\geq 1$, $0<\sigma<\min\{2,d\}$, $0<\alpha<\frac{4}{d}$ and $a>0$. Then, it holds that:
		\begin{itemize}
			\item The set $\mathcal{N}(a)$ is not empty.
			\item If $v \in \mathcal{N}(a)$, then there exists a positive radially symmetric function $\phi \in H^1$ such that $v(x) = e^{i\theta} \phi(x)$ for some $\theta \in \R$.
			\item The set $\mathcal{N}(a)$ is orbitally stable under the flow of the focusing problem \eqref{NLS-att}.
		\end{itemize}
	\end{theorem}
	
	In the mass-critical case, i.e. $\alpha=\frac{4}{d}$, under an appropriate assumption on $a$, the energy functional is bounded from below on $S(a)$. We have the following existence and stability of standing waves.
	\begin{theorem} \label{theo-exi-gro-mas-cri}
		Let $d\geq 1$, $0<\sigma<\min\{2,d\}$, $\alpha=\frac{4}{d}$ and $0<a<a^*:=\|Q\|^2_{L^2}$, where $Q$ is the unique (up to symmetries) positive radial solution to \eqref{ell-equ-mas-cri}. Then, it holds that:
		\begin{itemize}
			\item The set $\mathcal{N}(a)$ is not empty.
			\item If $v \in \mathcal{N}(a)$, then there exists a positive radially symmetric function $\phi \in H^1$ such that $v(x) = e^{i\theta} \phi(x)$ for some $\theta \in \R$.
			\item The set $\mathcal{N}(a)$ is orbitally stable under the flow of the focusing problem \eqref{NLS-att}.
		\end{itemize}
	\end{theorem}
	
	The proofs of Theorems $\ref{theo-exi-min-mas-sub}$ and $\ref{theo-exi-gro-mas-cri}$ are based on variational arguments using the radial compactness embedding. If we denote
	\[
	H^1_{\text{r}}:= \{ v\in H^1 \ : \ v \text{ is radially symmetric} \},
	\]
	then it is well-known that the embedding $H^1_{\text{r}} \hookrightarrow L^q$ is compact for any $2<q<\frac{2d}{d-2}$ if $d\geq 3$ ($2<q<\infty$ if $d=2$). Note that this compact embedding only holds in dimensions $d\geq 2$. The reason is that the inequality
	\[
	|v(x)| \leq C|x|^{\frac{1-d}{2}} \|v\|_{H^1}, \quad v\in H^1_{\text{r}}
	\]
	gives no decay in the case $d=1$. However, if $v$ is in addition radially decreasing, then it holds (see e.g. \cite[Appendix]{BL}) that 
	\begin{align} \label{rad-decrea-est}
	|v(x)| \leq \left( \frac{d}{|\Sb^{d-1}|}\right)^{\frac{1}{2}} |x|^{-\frac{d}{2}} \|v\|_{L^2}.
	\end{align}
	The above inequality yields the compact embedding
	\begin{align} \label{com-emb}
	H^1_{\text{rd}} \hookrightarrow L^q, \quad 2<q<\frac{2d}{d-2} \text{ if } d\geq 3 \ ( 2<q<\infty \text{ if } d=1,2),
	\end{align}
	where
	\[
	H^1_{\text{rd}}:= \{ v \in H^1_{\text{r}} \ : \ v \text{ is radially decreasing}\}.
	\]
	For the reader's convenience, we give the proof of \eqref{com-emb} in the Appendix. 
	
	\subsection{Blow-up behavior of standing waves}
	We next study the blow-up behavior of standing waves as the mass tends to a critical value in the mass-critical case. To our knowledge, the first paper addressed the blow-up behavior of standing waves in the mass-critical case belongs to Guo-Seiringer \cite{GS}. They studied the behavior of minimizers for 
	\[
	\min \{E_b(v) \ : \ v \in H^1, \|v\|_{L^2} =1\},
	\]
	where
	\[
	E_b(v) := \int_{\R^2} |\nabla v(x)|^2 dx + \int_{\R^2} V(x) |v(x)|^2 dx - \frac{b}{2} \int_{\R^2} |v(x)|^4 dx
	\]
	and $V$ is a trapping potential which has finite isolated minima. This result has been extended to ring-shaped trapping potentials in \cite{GZZ}, to periodic potentials in \cite{WZ} and to attractive potential vanishing at infinity in \cite{Phan}. In this paper, we have the following result.
	
	\begin{theorem} \label{theo-blo-up-behavior}
		Let $d\geq 1$, $0<\sigma<\min\{2,d\}$, $\alpha=\frac{4}{d}$ and $a^*=\|Q\|^2_{L^2}$, where $Q$ is the unique (up to symmetries) positive radial solution to \eqref{ell-equ-mas-cri}. Then, it holds that:
		\begin{itemize}
			\item If $a\geq a^*$, then there is no minimizer for $I(a)$.
			\item If $v_a$ is a non-negative minimizer for $I(a)$ with $0<a <a^*$, then $v_a$ blows up as $a \nearrow a^*$ in the sense that
			\begin{align} \label{blowup}
			\lim_{a \nearrow a^*} \|\nabla v_a\|_{L^2} = \infty.
			\end{align}
			Moreover, 
			\[
			\beta_a^{\frac{d}{2(2-\sigma)}} v_a\left( \beta_a^{\frac{1}{2-\sigma}} \cdot \right) \rightarrow \lambda_0^{\frac{d}{2}} Q(\lambda_0 \cdot) \text{ strongly in } H^1 \text{ as } a \nearrow a^*,
			\]
			where
			\begin{align} \label{def-bet-a}
			\beta_a:= 1-\left(\frac{a}{a^*}\right)^{\frac{2}{d}}, \quad \lambda_0 := \left( \frac{\sigma G(Q_0)}{d}\right)^{\frac{1}{2-\sigma}}, \quad Q_0 = \frac{Q}{\|Q\|_{L^2}}.
			\end{align}
		\end{itemize}
	\end{theorem}
	Note that since $\|\nabla |v|\|_{L^2} \leq \|\nabla v\|_{L^2}$, we can always assume that minimizers for $I(a)$ are non-negative. The proof is inspired by recent arguments of Phan \cite{Phan} as follows. The first step is to derive energy estimates for $I(a)$ (see \eqref{energy-estimate}). Using these estimates and a suitable change of variables, we show that the sequence of minimizers converges strongly in $H^1$ to an optimizer for the Gagliardo-Nirenberg (GN) inequality
	\begin{align} \label{sharp-GN-inequality}
	\|v\|^{\frac{4}{d}+2}_{L^{\frac{4}{d}+2}} \leq \frac{d+2}{d} \left(\frac{\|v\|_{L^2}}{\|Q\|_{L^2}}\right)^{\frac{4}{d}} \|\nabla v\|^2_{L^2}.
	\end{align}
	It then follows from the uniqueness (up to symmetries) of optimizers for the GN inequality that the limit equals to $Q$ modulo symmetries. Finally, we determine the exact limit by matching the energy.
	
	In the mass-supercritical case, i.e. $\alpha>\frac{4}{d}$, the energy functional is no longer bounded from below on $S(a)$. Indeed, let $v \in H^1$ be such that $\|v\|^2_{L^2} =a$. We define $v^\lambda(x) := \lambda^{\frac{d}{2}} v(\lambda x)$. It is clear that $\|v^\lambda\|^2_{L^2} = \|v\|^2_{L^2} =a$ and
	\[
	E(v^\lambda) = \frac{\lambda^2}{2} \|\nabla v\|^2_{L^2} -\frac{\lambda^\sigma}{2} G(v) - \frac{\lambda^\beta}{\alpha+2} \|v\|^{\alpha+2}_{L^{\alpha+2}},
	\]
	where $\beta$ is as in \eqref{def-bet}. Since  $\alpha>\frac{4}{d}$ or $\beta>2$, it follows that $E(v^\lambda)\rightarrow -\infty$ as $\lambda \rightarrow +\infty$. There is thus no minimizer for $I(a)$ in this case. Although there is no minimizers for $I(a)$, one may find normalized solutions for \eqref{ell-equ-att} in the mass-supercritical case by following a recent method of Bellazzini-Boussaid-Jeanjean-Visciglia \cite{BBJV}. The idea is to consider the local minimizing problem
	\[
	\inf\{E(v) \ : \ v \in S(a) \cap B(r)\},
	\]
	where
	\[
	B(r):= \{ v \in H^1 \ : \ H_0(v) := \|\nabla v\|^2_{L^2} - G(v) \leq r\}.
	\]
	This method works well for potentials $V$ satisfying
	\[
	\inf \left\{ \|\nabla v\|^2_{L^2} + \int V|v|^2 dx \ : \ \|v\|^2_{L^2}= 1\right\} >0,
	\]
	for instance, $V = |x|^2$ or $V= \sum_{j=1}^k x_j^2$, where $x= (x_1, \cdots, x_k, \cdots, x_d)$. In the case of attractive inverse-power potential, the minimum of the spectrum is negative, and the method of \cite{BBJV} is not directly applicable. 
	
	After the paper is submitted, the author was informed by Prof. Ohta that the blow-up result given in Theorem $\ref{theo-sha-threshold}$ is actually proved by Fukaya-Ohta in \cite{FO}. In fact, they proved that if $u_0 \in \mathcal{B}_\omega$ and $|x| u_0 \in L^2$, where
	\[
	\mathcal{B}_\omega := \{v \in H^1 \ : \ \|v\|_{L^2} \leq \|\phi\|_{L^2}, S_\omega(v) <S_\omega(\phi), \|v\|_{L^{\alpha+2}} > \|\phi\|_{L^{\alpha+2}}, Q(v)<0 \},
	\]
	then the corresponding solution to the focusing problem \eqref{NLS-att} blows up in finite time. Moreover, it is not hard to check that $\mathcal{K}^-_\omega = \mathcal{B}_\omega$. 
	
	This paper is organized as follows. In Section $\ref{S2}$, we prove the existence of ground states given in Proposition $\ref{prop-exi-min}$. In Section $\ref{S3}$, we give the proof of sharp threshold of global existence and blow-up for the focusing problem \eqref{NLS-att} given in Theorem $\ref{theo-sha-threshold}$. Section $\ref{S4}$ is devoted to the existence and stability of standing waves given in Theorems $\ref{theo-exi-min-mas-sub}$ and $\ref{theo-exi-gro-mas-cri}$. In Section $\ref{S5}$, we study the blow-up behavior of standing waves in the mass-critical case. Finally, the uniqueness of positive radial solutions to \eqref{ell-equ-att} is given in Appendix.

	\section{Existence of ground states}
	\label{S2}
	In this section, we prove the existence of ground states given in Proposition $\ref{prop-exi-min}$. To do so, we define the functional
	\[
	H_\omega(v):= \|\nabla v\|^2_{L^2} - G(v) + \omega\|v\|^2_{L^2}. 
	\]
	Thanks to \eqref{eig-est} and Hardy's inequality, we see that for $\omega>-\mu_1$ fixed,
	\begin{align} \label{equ-nor}
	H_\omega(v) \sim \|v\|^2_{H^1}.
	\end{align}
	More precisely, there exists $C>0$ such that
	\[
	\frac{\min\{C, \omega+\mu_1\}}{2} \|v\|^2_{H^1} \leq H_\omega(v) \leq \max\{1, \omega\} \|v\|^2_{H^1}.
	\]
	In fact, the upper bound follows easily from the fact $G(v)\geq 0$. To see the lower bound, we first have from \eqref{eig-est} that
	\begin{align} \label{lower-1}
	H_\omega(v) \geq (\omega+\mu_1) \|v\|^2_{L^2}. 
	\end{align}
	On the other hand, by Hardy's inequality (see e.g. \cite[Lemma 2.6]{ZZ})
	\begin{align} \label{hardy}
	\int |x|^{-\sigma} |v(x)|^2 dx \leq C\||\nabla|^{\sigma/2} v\|^2_{L^2} \leq C \|\nabla v\|^\sigma_{L^2} \|v\|^{2-\sigma}_{L^2}
	\end{align}
	and the fact $0<\sigma<2$, the Young inequality implies that
	\[
	G(v) \leq \frac{\sigma}{2} \|\nabla v\|^2_{L^2} + C \|v\|^2_{L^2}
	\]
	for some constant $C>0$. It follows that
	\begin{align} \label{lower-2}
	H_\omega(v) \geq \frac{2-\sigma}{2} \|\nabla v\|^2_{L^2} + (\omega-C)\|v\|^2_{L^2}
	\end{align}
	By choosing $\lambda$ such that $\omega-C + \lambda(\omega+\mu_1) \geq 0$, we infer from \eqref{lower-1} and \eqref{lower-2} that
	\[
	(1+\lambda) H_\omega(v) \geq \frac{2-\sigma}{2} \|\nabla v\|^2_{L^2} + (\omega - C+\lambda(\omega+\mu_1)) \|v\|^2_{L^2}
	\]
	hence
	\[
	H_\omega(v) \geq \frac{2-\sigma}{2(1+\lambda)} \|\nabla v\|^2_{L^2}.
	\]
	This together with \eqref{lower-1} imply
	\[
	2H_\omega(v) \geq \frac{2-\sigma}{2(1+\lambda)} \|\nabla v\|^2_{L^2} + (\omega+\mu_1) \|v\|^2_{L^2}
	\]
	which shows the lower bound.
	
	Note that the action functional can be rewritten as
	\begin{align} \label{act-fun}
	S_\omega(v) = \frac{1}{2} K_\omega(v) +\frac{\alpha}{2(\alpha+2)} \|v\|^{\alpha+2}_{L^{\alpha+2}} = \frac{1}{\alpha+2} K_\omega(v) + \frac{\alpha}{2(\alpha+2)} H_\omega(v).
	\end{align}
	
	\begin{lemma} \label{lem-non-empty-K-ome}
		Let $d\geq 1$, $0<\sigma<\min \{2,d\}$ and $0<\alpha<\frac{4}{d-2}$ if $d\geq 3$ ($0<\alpha<\infty$ if $d=1,2$). If $\omega>-\mu_1$, then there exists $v\in H^1\backslash \{0\}$ such that $K_\omega(v)=0$. In particular, the minimizing problem \eqref{d-ome} is well-defined.
	\end{lemma}
	\begin{proof}
		Let $v\in H^1 \backslash \{0\}$. If $K_\omega(v)=0$, we are done. If $K_\omega(v) \ne 0$, then for any $\lambda>0$,
		\[
		K_\omega(\lambda v) = \lambda^2 H_\omega(v) - \lambda^{\alpha+2} \|v\|^{\alpha+2}_{L^{\alpha+2}}.
		\]
		Note that since $\omega >-\mu_1$, by \eqref{eig-est}, $H_\omega(v) \geq (\omega+\mu_1) \|v\|^2_{L^2}>0$. It follows that
		$K_\omega(\lambda_0 v) =0$, where
		\[
		\lambda_0 = \left( \frac{H_\omega(v)}{\|v\|^{\alpha+2}_{L^{\alpha+2}}} \right)^{\frac{1}{\alpha}}>0.
		\]
		It closes the proof.
	\end{proof}
	
	\begin{lemma} \label{lem-pos-d-ome}
		$d(\omega)>0$.
	\end{lemma}
	
	\begin{proof}
		Let $v\in H^1 \backslash \{0\}$ be such that $K_\omega(v)=0$. Using \eqref{equ-nor} and the fact $H_\omega(v)= \|v\|^{\alpha+2}_{L^{\alpha+2}}$, the Sobolev embedding implies 
		\[
		\|v\|^2_{L^{\alpha+2}} \leq C_1 \|v\|^2_{H^1} \leq C_2 H_\omega(v) = C_2 \|v\|^{\alpha+2}_{L^{\alpha+2}},
		\]
		for some constants $C_1, C_2>0$. It follows that
		\[
		S_\omega(v) = \frac{\alpha}{2(\alpha+2)} \|v\|^{\alpha+2}_{L^{\alpha+2}} \geq \frac{\alpha}{2(\alpha+2)} \left(\frac{1}{C_2}\right)^{\frac{\alpha+2}{\alpha}}.
		\]
		The result follows by taking the infimum over $v \in H^1\backslash \{0\}$ with $K_\omega(v) =0$. 
	\end{proof}
	
	We denote the set of all minimizers for \eqref{d-ome} by
	\[
	\mathcal{M}_\omega:= \left\{ v \in H^1\backslash \{0\} \ : \ K_\omega(v) =0, \ S_\omega(v) = d(\omega) \right\}.
	\]
	It is well-known (see e.g. \cite{FO, FO-poten}) that if $\mathcal{M}_\omega$ is non-empty, then $\mathcal{M}_\omega \equiv \mathcal{G}_\omega$. In \cite{FO}, Fukaya-Ohta makes use of the weak continuity of the potential energy (see e.g. \cite[Theorem 11.4]{LL}) to show the non-emptiness of $\mathcal{M}_\omega$. In the following result, we give an alternative proof of this result.
	
	\begin{lemma} \label{lem-non-emp-M-ome}
		The set $\mathcal{M}_\omega$ is non-empty.
	\end{lemma}
	
	\begin{proof}
		We first observe from Lemma $\ref{lem-non-empty-K-ome}$ that if $v\in H^1\backslash \{0\}$ satisfying $K_\omega(v) \leq 0$, then there exists $\lambda_0 \in (0,1]$ such that $K_\omega(\lambda_0 v)=0$. 
		
		We next claim that any minimizing sequence for $d(\omega)$ can be chosen to be radially symmetric and radially decreasing. Indeed, let $(v_n)_{n\geq 1}$ be a minimizing sequence for $d(\omega)$. Let $v^*_n$ be the symmetric rearrangement of $v_n$. Note that the symmetric rearrangement preserves the $L^p$-norm and by Polya-Szego's inequality, $\|\nabla v^*_n\|_{L^2} \leq \|\nabla v_n\|_{L^2}$. We also have from the Hardy-Littlewood's inequality that 
		\[
		\int |x|^{-\sigma} |v_n(x)|^2 dx \leq \int |x|^{-\sigma} |v^*_n(x)|^2 dx.
		\]
		It follows that $H_\omega(v^*_n) \leq H_\omega(v_n)$ and $K_\omega(v_n^*) \leq K_\omega(v_n)=0$. By the above observation, there exists $(\mu_n)_{n\geq 1} \subset (0,1]$ such that $K_\omega(\mu_n v^*_n) =0$ for all $n\geq 1$. We have
		\[
		S_\omega(\mu_n v^*_n) = \frac{\alpha}{2(\alpha+2)} \mu_n^2 H_\omega(v_n^*) \leq \frac{\alpha}{2(\alpha+2)} H_\omega(v^*_n) \leq \frac{\alpha}{2(\alpha+2)} H_\omega(v_n) = S_\omega(v_n).
		\]
		This shows that $(\mu_n v^*_n)_{n\geq 1}$ is also a minimizing sequence for $d(\omega)$.
		
		We next show that any minimizing sequence for $d(\omega)$ is bounded in $H^1$. In fact, let $(v_n)_{n\geq 1}$ be a minimizing sequence for $d(\omega)$. It follows that $H_\omega(v_n) = \|v_n\|^{\alpha+2}_{L^{\alpha+2}}$ for all $n\geq 1$. By \eqref{act-fun}, we have
		\[
		S_\omega(v_n) = \frac{\alpha}{2(\alpha+2)} H_\omega(v_n) = \frac{\alpha}{2(\alpha+2)} \|v_n\|^{\alpha+2}_{L^{\alpha+2}} \rightarrow d(\omega)
		\]
		as $n\rightarrow \infty$. We infer that there exists $C>0$ such that
		\[
		H_\omega(v_n) \leq \frac{2(\alpha+2)}{\alpha} d(\omega) +C
		\]
		for all $n\geq 1$. Thanks to \eqref{equ-nor}, we see that $(v_n)_{n\geq 1}$ is a bounded sequence in $H^1$. 
		
		Now let $(v_n)_{n\geq 1}$ be a radially symmetric and radially decreasing minimizing sequence for $d(\omega)$. Since $(v_n)_{n\geq 1}$ is bounded in $H^1$, the compact embedding \eqref{com-emb} implies that there exist $v\in H^1$ and a subsequence still denoted by $(v_n)_{n\geq 1}$ such that $v_n \rightharpoonup v$ weakly in $H^1$ and $v_n \rightarrow v$ strongly in $L^q$ for any $2<q<\frac{2d}{d-2}$ if $d\geq 3$ ($2<q<\infty$ if $d=1,2$). This implies in particular that $v\ne 0$. Indeed, since $K_\omega(v_n) =0$ for all $n\geq 1$, by the same argument as in Lemma $\ref{lem-pos-d-ome}$, there exists $C>0$ such that $\|v_n\|_{L^{\alpha+2}} \geq C$. By the strong convergence, we get $\|v\|_{L^{\alpha+2}} \geq C>0$.
		
		We next claim that $G(v_n) \rightarrow G(v)$ as $n\rightarrow \infty$. To see this, we estimate
		\begin{align*}
		|G(v_n) -G(v)| &\leq c \int |x|^{-\sigma} ||v_n(x)|-|v(x)||(|v_n|+|v(x)|) dx \\
		&\leq c \int |x|^{-\sigma} |v_n(x) -v(x)| (|v_n(x)| + |v(x)|) dx \\
		&= c \left[\int_{B(0,1)}  + \int_{B^c(0,1)} \right] |x|^{-\sigma} |v_n(x) -v(x)| (|v_n(x)| + |v(x)|) dx =: I_1 + I_2,
		\end{align*}
		where $B(0,1)$ the unit ball in $\R^d$ and $B^c(0,1)$ is its complement.
		
		On $B(0,1)$, we have
		\begin{align*}
		I_1 \lesssim \||x|^{-\sigma}\|_{L^\gamma(B(0,1))} \||v_n-v|(|v_n|+|v|)\|_{L^\mu} \lesssim \|v_n-v\|_{L^\delta} (\|v_n\|_{L^\tau} + \|v\|_{L^\tau}),
		\end{align*}
		where $\gamma, \mu, \delta, \tau\geq 1$ satisfy 
		\[
		1=\frac{1}{\gamma}+ \frac{1}{\mu}, \quad \frac{d}{\gamma}>\sigma, \quad \frac{1}{\mu}=\frac{1}{\delta} + \frac{1}{\tau}.
		\]
		Here the second condition ensures that $\||x|^{-\sigma}\|_{L^\gamma(B(0,1))}<\infty$. Using the fact $v_n \rightarrow v$ for $2<q<\frac{2d}{d-2}$ if $d\geq 3$ ($2<q<\infty$ if $d=1,2$) and the Sobolev embedding $H^1 \subset L^q$, we see that if we are able to choose $2<\delta, \tau<\frac{2d}{d-2}$ in the case $d\geq 3$ ($2<\delta, \tau<\infty$ in the case $d=1,2$) so that
		\begin{align} \label{con-del-tau}
		\frac{1}{\delta} + \frac{1}{\tau} <\frac{d-\sigma}{d}
		\end{align}
		then $I_1 \rightarrow 0$ as $n\rightarrow \infty$. The condition \eqref{con-del-tau} is fulfilled if we take $\delta=\tau=\frac{2d}{d-2}- \varep$ with $0<\varep<\frac{2d(2-\sigma)}{(d-2)(d-\sigma)}$ in the case $d\geq 3$ and $\delta = \tau$ large enough in the case $d=1,2$. 
		
		On $B^c(0,1)$, we estimate
		\begin{align*}
		I_2 \lesssim \||x|^{-\sigma}\|_{L^\gamma(B^c(0,1))} \||v_n-v|(|v_n|+|v|)\|_{L^\mu} \lesssim \|v_n-v\|_{L^\delta} (\|v_n\|_{L^\tau} + \|v\|_{L^\tau}),
		\end{align*}
		where $\gamma, \mu, \delta, \tau\geq 1$ satisfy 
		\[
		1=\frac{1}{\gamma}+ \frac{1}{\mu}, \quad \frac{d}{\gamma}<\sigma, \quad \frac{1}{\mu}=\frac{1}{\delta} + \frac{1}{\tau}.
		\]
		If we choose $2<\delta, \tau<\frac{2d}{d-2}$ in the case $d\geq 3$ ($2<\delta,\tau<\infty$ in the case $d=1,2$) so that
		\[
		\frac{1}{\delta}+\frac{1}{\tau} >\frac{d-\sigma}{d},
		\]
		then $I_2 \rightarrow 0$ as $n\rightarrow \infty$. The above condition is satisfied for $\delta=\tau=2+\varep$ with $0<\varep<\frac{2\sigma}{d-\sigma}$.
		
		Combining the above two cases, we prove that $G(v_n) \rightarrow G(v)$ as $n\rightarrow \infty$. It follows that
		\[
		K_\omega(v) \leq \liminf_{n\rightarrow \infty} K_\omega(v_n)=0.
		\]
		There thus exists $\lambda_0 \in (0,1]$ such that $K_\omega(\lambda_0 v) =0$. By the definition of $d(\omega)$,
		\[
		d(\omega) \leq S_\omega(\lambda_0 v) = \frac{\alpha}{2(\alpha+2)} \lambda_0^2 H_\omega(v) \leq \frac{\alpha}{2(\alpha+2)} \liminf_{n\rightarrow \infty} H_\omega(v_n) = \liminf_{n\rightarrow \infty} S_\omega(v_n) = d(\omega).
		\]
		This implies that $S_\omega(\lambda_0 v) = d(\omega)$ or $\lambda_0 v$ is a minimizer for $d(\omega)$. Moreover, all inequalities above are in fact equalities, that is, $\lambda_0 =1$, $K_\omega(v)=0$ and $H_\omega(v)= \lim_{n \rightarrow \infty} H_\omega(v_n)$ which implies by \eqref{equ-nor} that $v_n \rightarrow v$ strongly in $H^1$. The proof is complete.
	\end{proof}

	\section{Sharp threshold for global existence and blow-up}
	\label{S3}		
	In this section, we prove the sharp threshold of global existence and blow-up for the focusing problem \eqref{NLS-att}. To this end, we set 
	\begin{align} \label{scaling}
	v^\lambda(x):= \lambda^{\frac{d}{2}} v(\lambda x), \quad \lambda>0.
	\end{align}
	A direct computation shows
	\begin{align} \label{pro-scaling}
	\begin{aligned}
	\|v^\lambda\|^2_{L^2} &= \|v\|^2_{L^2}, & \|\nabla v^\lambda\|^2_{L^2} &= \lambda^2 \|\nabla v\|^2_{L^2}, \\
	G(v^\lambda) &= \lambda^\sigma G(v), & \|v^\lambda\|^{\alpha+2}_{L^{\alpha+2}} &= \lambda^{\beta} \|v\|^{\alpha+2}_{L^{\alpha+2}},
	\end{aligned}
	\end{align}
	where $\beta$ is given in \eqref{def-bet}. It follows that
	\begin{align} \label{S-ome-sca}
	S_\omega(v^\lambda) = \frac{\lambda^2}{2} \|\nabla v\|^2_{L^2} - \frac{\lambda^\sigma}{2} G(v) + \frac{\omega}{2} \|v\|^2_{L^2} - \frac{\lambda^\beta}{\alpha+2} \|v\|^{\alpha+2}_{L^{\alpha+2}},
	\end{align}
	and
	\begin{align} \label{I-ome-sca}
	Q(v^\lambda) = \lambda^2 \|\nabla v\|^2_{L^2} -\frac{\sigma}{2} \lambda^\sigma G(v) - \frac{\beta}{\alpha+2} \lambda^\beta \|v\|^{\alpha+2}_{L^{\alpha+2}} = \lambda \partial_\lambda S_\omega(v^\lambda).
	\end{align}
	In particular, $Q(v)= \left. \partial_\lambda S_\omega(v^\lambda)\right|_{\lambda=1}$.
	
	\begin{lemma} \label{lem-poh-ide}
		Let $d\geq 1$, $0<\sigma<\min \{2,d\}$, $0<\alpha<\frac{4}{d-2}$ if $d\geq 3$ ($0<\alpha<\infty$ if $d=1,2$) and $\omega>-\mu_1$. Let $\phi \in H^1 \backslash \{0\}$ be a solution to \eqref{ell-equ-att}. Then it holds that
		\begin{align} \label{poh-ide}
		\begin{aligned}
		\|\nabla \phi\|^2_{L^2} &+ \omega \|\phi\|^2_{L^2} - G(\phi) - \|\phi\|^{\alpha+2}_{L^{\alpha+2}} =0, \\
		\frac{2-d}{2} \|\nabla \phi\|^2_{L^2} &- \frac{d\omega}{2} \|\phi\|^2_{L^2} + \frac{(d-\sigma)}{2} G(\phi) + \frac{d}{\alpha+2} \|\phi\|^{\alpha+2}_{L^{\alpha+2}}=0. 
		\end{aligned}
		\end{align}
		In particular,  $K_\omega(\phi)=Q(\phi)=0$.
	\end{lemma}
	
	\begin{proof}
		By multiplying both sides of \eqref{ell-equ-att} with $\overline{\phi}$ and integrating over $\R^d$, we get the first identity in \eqref{poh-ide} which is $K_\omega(\phi)=0$. Multiplying both sides of \eqref{ell-equ-att} with $x \cdot \nabla \overline{\phi}$, integrating over $\R^d$ and taking the real part, we obtain the second identity in \eqref{poh-ide}. Note that we only make formal computations here. Due to the singularity of the inverse-power potential at zero, we need to integrate on the annulus $\{ x\in \R^d \ : \ r \leq |x| \leq R\}$ for $R>r>0$ and then take the limit as $R\rightarrow +\infty$ and $r\rightarrow 0$. We refer the reader to \cite[Lemma 3.2]{Dinh-ins} for detailed computations in the case of inverse-square potential. Multiplying both sides of the first identity with $\frac{d}{2}$ and adding to the second identity, we obtain $Q(v)=0$. The proof is complete.
	\end{proof}

	\begin{proposition} \label{prop-key-est}
		Let $d\geq 1$, $0<\sigma<\min \{2,d\}$, $\frac{4}{d}<\alpha<\frac{4}{d-2}$ if $d\geq 3$ ($\frac{4}{d}<\alpha<\infty$ if $d=1,2$) and $\omega>-\mu_1$. Let $\phi \in \mathcal{G}_\omega$ be such that $\left. \partial^2_\lambda S_\omega(\phi^\lambda) \right|_{\lambda=1} \leq 0$, where $\phi^\lambda$ is as in \eqref{scaling}. Let $v \in H^1 \backslash \{0\}$ be such that 
		\[
		\|v\|_{L^2} \leq \|\phi\|_{L^2}, \quad K_\omega(v) \leq 0, \quad Q(v) \leq 0.
		\]
		Then it holds that
		\[
		Q(v) \leq 2(S_\omega(v) - S_\omega(\phi)).
		\]
	\end{proposition}
	
	\begin{remark} \label{rem-intersection}
		It is easy to see from Proposition $\ref{prop-key-est}$ that for $\phi \in \mathcal{G}_\omega$ satisfying $\left. \partial^2_\lambda S_\omega(\phi^\lambda) \right|_{\lambda=1} \leq 0$, 
		\begin{align} \label{empty-set}
		\left\{v\in H^1 \backslash \{0\} \ : \ \|v\|_{L^2} \leq \|\phi\|_{L^2}, S_\omega(v) < S_\omega(\phi), K_\omega(v)<0, Q(v)=0 \right\} = \emptyset,
		\end{align}
		Indeed, if there exists $v\in H^1 \backslash \{0\}$ satisfying $\|v\|_{L^2} \leq \|\phi\|_{L^2}, S_\omega(v) < S_\omega(\phi), K_\omega(v)<0$ and $Q(v)=0$, then by Proposition $\ref{prop-key-est}$,
		\[
		0 = Q(v) \leq 2 (S_\omega(v) - S_\omega(\phi)) <0
		\]
		which is a contradiction.
	\end{remark}

	\noindent {\it Proof of Proposition $\ref{prop-key-est}$.} The proof is similar to \cite[Lemma 3.2]{FO}, where $K_\omega(v) \leq 0$ is replaced by $\|v\|_{L^{\alpha+2}} \geq \|\phi\|_{L^{\alpha+2}}$. For the reader's convenience, we give some details. If $K_\omega(v)=0$, then by Proposition $\ref{prop-exi-gro-sta}$ and $Q(v)\leq 0$, we have $S_\omega(\phi) \leq S_\omega(v) \leq S_\omega(v) - \frac{1}{2}Q(v)$. Suppose that $K_\omega(v)<0$. Let $v^\lambda$ be as in \eqref{scaling} and define 
	\begin{align*}
	f(\lambda)&:= S_\omega(v^\lambda) - \frac{\lambda^2}{2} Q(v) = \frac{1}{2} \left(\frac{\sigma}{2}\lambda^2-\lambda^\sigma\right) G(v) + \frac{1}{\alpha+2} \left( \frac{\beta}{2} \lambda^2-\lambda^\beta\right) \|v\|^{\alpha+2}_{L^{\alpha+2}} + \frac{\omega}{2} \|v\|^2_{L^2}.
	\end{align*}
	We have
	\begin{align*}
	K_\omega(v^\lambda) &:= \lambda^2 \|\nabla v\|^2_{L^2} - \lambda^\sigma G(v) + \omega \|v\|^2_{L^2} - \lambda^\beta \|v\|^{\alpha+2}_{L^{\alpha+2}}.
	\end{align*}
	Since $\lim_{\lambda \rightarrow 0} K_\omega(v^\lambda)= \omega \|v\|^2_{L^2}>0$ and $K_\omega(v)<0$, it follows that there exists $\lambda_0 \in (0,1)$ such that $K_\omega(v^{\lambda_0})=0$. If we have $f(\lambda_0) \leq f(1)$, then it follows from $Q(v) \leq 0$ that 
	\[
	S_\omega(\phi) \leq S_\omega(v^{\lambda_0}) \leq S_\omega(v^{\lambda_0}) - \frac{\lambda_0^2}{2} Q(v) \leq S_\omega(v) -\frac{1}{2} Q(v).
	\]
	It remains to prove $f(\lambda_0) \leq f(1)$ which is in turn equivalent to
	\begin{align} \label{equ-con}
	G(v) \leq \frac{2(2\lambda_0^\beta - \beta\lambda_0^2 -2 +\beta)}{(\alpha+2)(\sigma \lambda_0^2 - 2 \lambda_0^\sigma -\sigma +2)} \|v\|^{\alpha+2}_{L^{\alpha+2}}.
	\end{align}
	Note that by \eqref{S-ome-sca}, the condition $\left. \partial^2_\lambda S_\omega(\phi^\lambda) \right|_{\lambda=1} \leq 0$ is equivalent to
	\begin{align} \label{con-phi}
	\|\nabla \phi\|^2_{L^2} - \frac{\sigma(\sigma-1)}{2} G(\phi) - \frac{\beta}{\alpha+2} (\beta-1) \|\phi\|^{\alpha+2}_{L^{\alpha+2}} \leq 0.
	\end{align}
	Since $K_\omega(\phi)=Q(\phi)=0$, we have from $\sigma K_\omega(\phi)-(\sigma+1)Q(\phi)=0$ and \eqref{con-phi} that
	\begin{align*}
	\omega \sigma \|\phi\|^2_{L^2} &= \|\nabla \phi\|^2_{L^2} - \frac{\sigma(\sigma-1)}{2} G(\phi) + \left( \sigma - \frac{\beta(\sigma+1)}{\alpha+2} \right) \|\phi\|^{\alpha+2}_{L^{\alpha+2}}  \\
	&\leq \left(\sigma + \frac{\beta(\beta -\sigma-2)}{\alpha+2}  \right) \|\phi\|^{\alpha+2}_{L^{\alpha+2}}. 
	\end{align*}
	In particular,
	\begin{align} \label{est-L2}
	\omega \|\phi\|_{L^2}^2 \leq \left( 1+ \frac{\beta(\beta-\sigma-2)}{\sigma(\alpha+2)} \right) \|\phi\|^{\alpha+2}_{L^{\alpha+2}}. 
	\end{align}
	Since $K_\omega(v^{\lambda_0}) =0$, by Proposition $\ref{prop-exi-gro-sta}$, we have
	\[
	\frac{\alpha}{2(\alpha+2)} \|\phi\|^{\alpha+2}_{L^{\alpha+2}} = S_\omega(\phi) \leq S_\omega(v^{\lambda_0}) = \frac{\alpha}{2(\alpha+2)} \|v^{\lambda_0}\|^{\alpha+2}_{L^{\alpha+2}} = \frac{\alpha}{2(\alpha+2)} \lambda_0^\beta \|v\|^{\alpha+2}_{L^{\alpha+2}}.
	\]
	We thus get from \eqref{est-L2} and the assumption $\|v\|_{L^2} \leq \|\phi\|_{L^2}$ that
	\begin{align} \label{est-L2-v}
	\omega \|v\|^2_{L^2} \leq \left( 1+ \frac{\beta(\beta-\sigma-2)}{\sigma(\alpha+2)} \right) \lambda_0^\beta \|v\|^{\alpha+2}_{L^{\alpha+2}}.
	\end{align}
	We also have from $K_\omega(v^{\lambda_0}) =0$, \eqref{est-L2-v} and $Q(v) \leq 0$ that
	\begin{align*}
	G(v) &= \lambda_0^{2-\sigma} \|\nabla v\|^2_{L^2} + \lambda_0^{-\sigma} \omega \|v\|^2_{L^2} - \lambda_0^{\beta-\sigma} \|v\|^{\alpha+2}_{L^{\alpha+2}} \\
	& \leq \lambda_0^{2-\sigma} \|\nabla v\|^2_{L^2} + \frac{\beta(\beta-\sigma-2)}{\sigma(\alpha+2)} \lambda_0^{\beta-\sigma} \|v\|^{\alpha+2}_{L^{\alpha+2}} \\
	&\leq \frac{\sigma}{2} \lambda_0^{2-\sigma} G(v) + \frac{\beta}{\alpha+2} \left( \lambda_0^{2-\sigma} + \frac{\beta-\sigma-2}{\sigma} \lambda_0^{\beta-\sigma} \right) \|v\|^{\alpha+2}_{L^{\alpha+2}}.
	\end{align*}
	Thus,
	\begin{align} \label{est-G}
	G(v) \leq \frac{2\beta}{(\alpha+2)(2-\sigma \lambda_0^{2-\sigma})} \left( \lambda_0^{2-\sigma} + \frac{\beta-\sigma-2}{\sigma} \lambda_0^{\beta-\sigma} \right) \|v\|^{\alpha+2}_{L^{\alpha+2}}.
	\end{align}
	In view of \eqref{equ-con} and \eqref{est-G}, it suffices to show that
	\begin{align*} 
	\frac{\beta}{2-\sigma \lambda_0^{2-\sigma}} \left( \lambda_0^{2-\sigma} + \frac{\beta-\sigma-2}{\sigma} \lambda_0^{\beta-\sigma} \right) \leq \frac{2\lambda_0^\beta - \beta\lambda_0^2 -2 +\beta}{\sigma \lambda_0^2 - 2 \lambda_0^\sigma -\sigma +2}
	\end{align*}
	which is equivalent to
	\[
	g(\lambda_0) := \frac{(2-\sigma \lambda_0^{2-\sigma}) (2 \lambda^\beta_0- \beta \lambda_0^2 - 2 +\beta)}{ \beta \lambda_0^{\beta-\sigma} ( \sigma \lambda_0^2 - 2 \lambda_0^\sigma - \sigma +2)} - \frac{1}{\lambda_0^{\beta-2}} - \frac{\beta-\sigma-2}{\sigma} \geq 0.
	\]
	Since $\lim_{\lambda\rightarrow 1} g(\lambda) =0$, the above inequality follows if we have
	\begin{align*}
	g'(\lambda) = \frac{2(1-\lambda^{2-\sigma})}{\beta \lambda^{\beta-\sigma+1}( \sigma \lambda^2- 2\lambda^\sigma - \sigma +2)^2} &\left[2\sigma(2-\sigma) \lambda^\beta - \sigma \beta (\beta-\sigma) \lambda^2 \right.\\
	&\left.+ 2\beta(\beta-2) \lambda^\sigma - (\beta-\sigma) (\beta-2) (2-\sigma) \right] \leq 0
	\end{align*}
	for all $\lambda \in (0,1)$. The above inequality holds if we have
	\[
	g_1(\lambda) := 2\sigma(2-\sigma) \lambda^\beta - \sigma \beta (\beta-\sigma) \lambda^2 + 2\beta(\beta-2) \lambda^\sigma - (\beta-\sigma) (\beta-2) (2-\sigma) \leq 0
	\]
	for all $\lambda \in (0,1)$. Since $g_1(1) =0$, it is enough to show that
	\[
	g'_1(\lambda) = 2\sigma \beta \lambda^{\sigma-1} [ (2-\sigma) \lambda^{\beta-\sigma} - (\beta-\sigma) \lambda^{2-\sigma} + \beta-2] \geq 0
	\]
	for all $\lambda \in (0,1)$. This is equivalent to
	\[
	g_2(\lambda):= (2-\sigma) \lambda^{\beta-\sigma} - (\beta-\sigma) \lambda^{2-\sigma} + \beta-2 \geq 0
	\]
	for all $\lambda \in (0,1)$. Since $g_2(1)=0$ and 
	\[
	g'_2(\lambda) = - (\beta-\sigma)(2-\sigma) \lambda^{1-\sigma} (1-\lambda^{\beta-2}) \leq 0,
	\]
	we have $g_2(\lambda) \geq 0$ for all $\lambda \in (0,1)$. Therefore, we obtain $f(\lambda_0) \leq f(1)$ and the proof is complete.
	\hfill $\Box$
	
	\begin{lemma} \label{lem-inv-set}
		Let $d\geq 1$, $0<\sigma <\min \{2,d\}$, $\frac{4}{d} <\alpha <\frac{4}{d-2}$ if $d\geq 3$ ($\frac{4}{d} <\alpha<\infty$ if $d=1,2$) and $\omega>-\mu_1$. Let $\phi \in \mathcal{G}_\omega$ be such that $\left. \partial^2_\lambda S_\omega(\phi^\lambda)\right|_{\lambda=1} \leq 0$, where $\phi^\lambda$ is as in \eqref{scaling}. Then the sets $\mathcal{K}^\pm_\omega$ are invariant under the flow of the focusing problem \eqref{NLS-att}.
	\end{lemma}
	
	\begin{proof}
		We only consider the case $\mathcal{K}^-_\omega$, the one for $\mathcal{K}^+_\omega$ is similar. Let $u_0 \in \mathcal{K}^-_\omega$, i.e. $\|u_0\|_{L^2} \leq \|\phi\|_{L^2}$, $S_\omega(u_0) <S_\omega(\phi)$, $K_\omega(u_0) <0$ and $Q(u_0)<0$. We will show that $u(t) \in \mathcal{K}^-_\omega$ for any $t$ in the existence time. By the conservation of mass and energy, we have 
		\begin{align} \label{inv-set-proof}
		\|u(t)\|_{L^2} = \|u_0\|_{L^2} \leq \|\phi\|_{L^2}, \quad S_\omega(u(t)) = S_\omega(u_0) <S_\omega(\phi)
		\end{align}
		for any $t \in I_{\text{max}}$, where $I_{\text{max}}$ is the maximal existence time interval. Let us prove $K_\omega(u(t))<0$ for any $t\in I_{\text{max}}$. Suppose that there exists $t_0 \in I_{\text{max}}$ such that $K_\omega(u(t_0)) \geq 0$. By the continuity of $t \mapsto K_\omega(u(t))$, there exists $t_1 \in I_{\text{max}}$ such that $K_\omega(u(t_1)) =0$. By Proposition $\ref{prop-exi-gro-sta}$, $S_\omega(u(t_1)) \geq S_\omega(\phi)$ which contradicts to \eqref{inv-set-proof}. We finally prove $Q(u(t))<0$ for any $t\in I_{\text{max}}$. Suppose it is not true, then there exists $t_2 \in I_{\text{max}}$ such that $Q(u(t_2)) \geq 0$. By continuity of $t \mapsto Q(u(t))$, there exists $t_3 \in I_{\text{max}}$ such that $Q(u(t_3)) =0$. We thus obtain a function $u(t_3) \in H^1 \backslash \{0\}$ satisfying $\|u(t_3)\|_{L^2} \leq \|\phi\|_{L^2}, S_\omega(u(t_3))<S_\omega(\phi), K_\omega(u(t_3))<0$ and $Q(u(t_3))=0$. This is not possible due to \eqref{empty-set}. The proof is complete.
	\end{proof}
	
	\begin{proposition} \label{prop-sharp-threshold}
		Let $d\geq 1$, $0<\sigma<\min\{2,d\}$, $\frac{4}{d}<\alpha<\frac{4}{d-2}$ if $d\geq 3$ ($\frac{4}{d}<\alpha<\infty$ if $d=1,2$) and $\omega>-\mu_1$. Let $\phi \in \mathcal{G}_\omega$ be such that $\left.\partial^2_\lambda S_\omega(\phi^\lambda)\right|_{\lambda=1} \leq 0$, where $\phi^\lambda$ is as in \eqref{scaling}. 
		\begin{itemize}
			\item If $u_0 \in \mathcal{K}^-_\omega$ and $|x|u_0 \in L^2$, then the corresponding solution to \eqref{NLS-att} blows up in finite time.
			\item If $u_0 \in \mathcal{K}^+_\omega$, then the corresponding solution to \eqref{NLS-att} exists globally in time.
		\end{itemize}
	\end{proposition}
	
	\begin{proof}
	Let us first consider the case $u_0 \in \mathcal{K}^-_\omega$ and $|x|u_0 \in L^2$. It is well-known that $|x|u(t) \in L^2$ for all $t\in I_{\text{max}}$. By \eqref{virial-action} and the convexity argument of Glassey \cite{Glassey}, it suffices to show there exists $\delta>0$ such that 
	\begin{align} \label{blowup-delta}
	Q(u(t)) \leq -\delta, \quad \forall t\in I_{\text{max}}.
	\end{align}
	To do so, we note that since $\mathcal{K}^-_\omega$ is invariant under the flow of \eqref{NLS-att}, $u(t) \in \mathcal{K}^-_\omega$ for all $t\in I_{\text{max}}$, i.e. $\|u(t)\|_{L^2} \leq \|\phi\|_{L^2}$, $S_\omega(u(t)) <S_\omega(\phi)$, $K_\omega(u(t)) <0$ and $Q(u(t))<0$ for all $t\in I_{\text{max}}$. Applying Proposition $\ref{prop-key-est}$ to $u(t)$, we get
	\[
	Q(u(t)) \leq 2(S_\omega(u(t)) - S_\omega(\phi)) = 2(S_\omega(u_0) - S_\omega(\phi)),
	\]
	for all $t \in I_{\text{max}}$, where we have used the conservation of mass and energy. This shows \eqref{blowup-delta} with $\delta:= 2(S_\omega(\phi) - S_\omega(u_0))>0$.

	We now consider $u_0 \in \mathcal{K}^+_\omega$. By the local well-posedness, it suffices to show there exists $C>0$ such that
	\begin{align} \label{uniform-H1-bound}
	\|u(t)\|_{H^1} \leq C, \quad \forall t \in I_{\text{max}}.
	\end{align}
	By Lemma $\ref{lem-inv-set}$, $u(t) \in \mathcal{K}^+_\omega$ for all $t\in I_{\text{max}}$, i.e. $\|u(t)\|_{L^2} \leq \|\phi\|_{L^2}, S_\omega(u(t)) <S_\omega(\phi)$, $K_\omega(u(t)) <0$ and $Q(u(t))>0$ for all $t\in I_{\text{max}}$. We have
	\begin{align*}
	\frac{1}{2} H_\omega(u(t)) - \frac{1}{\beta} Q(u(t)) = \left(\frac{1}{2} - \frac{1}{\beta}\right) \|\nabla u(t)\|^2_{L^2} + \frac{\omega}{2} \|u(t)\|^2_{L^2} - \frac{\beta-\sigma}{2\beta} G(u(t)) + \frac{1}{\alpha+2} \|u(t)\|^{\alpha+2}_{L^{\alpha+2}}.
	\end{align*}
	It follows that
	\begin{align*}
	\left(\frac{1}{2} - \frac{1}{\beta}\right) \|\nabla u(t)\|^2_{L^2} &= \frac{1}{2} H_\omega(u(t)) - \frac{1}{\alpha+2} \|u(t)\|^{\alpha+2}_{L^{\alpha+2}} -\frac{1}{\beta} Q(u(t)) - \frac{\omega}{2} \|u(t)\|^2_{L^2} +\frac{\beta-\sigma}{2\beta} G(u(t)) \\
	&=S_\omega(u(t)) - \frac{1}{\beta} Q(u(t)) - \frac{\omega}{2} \|u(t)\|^2_{L^2} +\frac{\beta-\sigma}{2\beta} G(u(t)) \\
	&<S_\omega(\phi) - \frac{\omega}{2} \|u(t)\|^2_{L^2} +\frac{\beta-\sigma}{2\beta} G(u(t)).
	\end{align*}
	By Hardy's inequality \eqref{hardy} and the fact $0<\sigma<2$, we apply the Young's inequality to have for any $\varep>0$,
	\begin{align} \label{har-ine-app}
	G(u(t)) \leq \varep \|\nabla u(t)\|^2_{L^2} + C(\varep) \|u(t)\|^2_{L^2}. 
	\end{align}
	We thus have that
	\[
	\left(\frac{1}{2} - \frac{1}{\beta} -\varep\right) \|\nabla u(t)\|^2_{L^2} < S_\omega(\phi) + \left(\frac{(\beta-\sigma) C(\varep)}{2\beta} - \frac{\omega}{2} \right) \|u(t)\|^2_{L^2}
	\]
	for all $t \in I_{\text{max}}$. Since we are considering the $L^2$-supercritical case, we see that $\beta>2$. By choosing $0<\varep<\frac{1}{2} -\frac{1}{\beta}$ and using the conservation of mass, we prove \eqref{uniform-H1-bound}. The proof is complete.
	\end{proof}
	
	\begin{remark}
		It is expected that the same finite time blow-up holds for radially symmetric initial data in $\mathcal{K}^-_\omega$. However, in the presence of inverse-power potentials with $0<\sigma<2$, it is not clear how to show it at the moment. In fact, by radial Sobolev embeddings (see e.g. \cite{Dinh-rep}), it suffices to show that for $\varep>0$ small enough, there exists $\delta(\varep)>0$ such that
		\begin{align} \label{blow-up-condition-radial}
		Q(u(t)) + \varep \|\nabla u(t)\|^2_{L^2} \leq -\delta(\varep)
		\end{align}
		for any $t \in I_{\text{max}}$. Note that
		\begin{align*}
		S_\omega(v) =\frac{1}{\beta} Q(v) + \frac{\beta-2}{2\beta} \|\nabla v\|^2_{L^2} + \frac{\omega}{2} \|v\|^2_{L^2} - \frac{\beta-2}{2\beta} G(v) =: \frac{1}{\beta} Q(v) + P_\omega(v)
		\end{align*}
		and
		\[
		\|\nabla v\|^2_{L^2} = \frac{2\beta}{\beta-2} P_\omega(v) -\frac{\omega \beta}{\beta-2} \|v\|^2_{L^2} + \frac{\beta-\sigma}{\beta-2} G(v).
		\]
		It follows that
		\begin{align*}
		Q(v) + \varep \|\nabla v\|^2_{L^2} &= \beta S_\omega(v) - \beta P_\omega(v) + \varep \|\nabla v\|^2_{L^2} \\
		&= \beta S_\omega(v) - \beta \left(1-\frac{2\varep}{\beta-2} \right) P_\omega(v) -\frac{\omega \beta \varep}{\beta-2} \|v\|^2_{L^2} + \frac{(\beta-\sigma)\varep}{\beta-2} G(v) \\
		&\leq \beta (1-\rho) S_\omega(\phi) - \beta \left(1-\frac{2\varep}{\beta-2} \right) P_\omega(v) + \frac{(\beta-\sigma)\varep}{\beta-2} G(v),
		\end{align*}
		where we have use $S_\omega(v) \leq (1-\rho)S_\omega(\phi)$ for some $\rho>0$ due to $S_\omega(v) <S_\omega(\phi)$. In the case of no potential, i.e. $c=0$ or $G(v)=0$, we can show that 
		\begin{align} \label{S-ome-no-potential}
		S_\omega(\phi) = \inf \{P_\omega(v) \ : \ v \in H^1 \backslash \{0\}, Q(v) =0\}.
		\end{align}
		Note that since $Q(v) <0$, there exists $\lambda_0 \in (0,1)$ such that $Q(\lambda_0 v)=0$. Thus
		\[
		S_\omega(\phi) \leq P_\omega(\lambda_0 v) = \lambda_0^2 P_\omega(v) < P_\omega(v).
		\]
		This shows that
		\[
		Q(v) + \varep\|\nabla v\|^2_{L^2} \leq -\beta \left( \rho - \frac{2\varep}{\beta-2}\right) S_\omega(\phi),
		\]
		and \eqref{blow-up-condition-radial} holds with $\delta(\varep) = \beta \left( \rho - \frac{2\varep}{\beta-2}\right) S_\omega(\phi)>0$. In the case of inverse-power potentials with $0<\sigma<2$, we do not have \eqref{S-ome-no-potential}, and there is an additional positive term $\frac{(\beta-\sigma)\varep}{\beta-2} G(v)$ which is difficult to control.
	\end{remark}

	\begin{lemma}[\cite{FO-poten}] \label{lem-large-omega}
		There exists $\omega_0>-\mu_1$ such that if $\omega \geq \omega_0$ and $\phi_\omega \in \mathcal{G}_\omega$, then $\left. \partial^2_\lambda S_\omega(\phi^\lambda_\omega) \right|_{\lambda=1} \leq 0$, where $\phi^\lambda_\omega$ is as in \eqref{scaling}.
	\end{lemma}
	We refer the readers to \cite[Section 2]{FO-poten} for the proof of this result.

	\noindent {\it Proof of Theorem $\ref{theo-sha-threshold}$.}
	It follows directly from Proposition $\ref{prop-sharp-threshold}$ and Lemma $\ref{lem-large-omega}$.	
	\hfill $\Box$

	\section{Existence and stability of standing waves}
	\label{S4}
	In this section, we give the proof of Theorem $\ref{theo-exi-min-mas-sub}$ and $\ref{theo-exi-gro-mas-cri}$.
	
	\noindent {\it Proof of Theorem $\ref{theo-exi-min-mas-sub}$.} The proof is divided in several steps. 
		
	\noindent {\bf Step 1.} We will show that the minimizing problem \eqref{min-prob-a} is well-defined and there exists $C>0$ such that $I(a) \leq -C <0$. Indeed, let $v \in H^1$ be such that $\|v\|^2_{L^2}=a$. By Hardy's inequality and Young's inequality with $0<\sigma<2$ (see \eqref{har-ine-app}), we have for any $\varep>0$,
\[
G(v) \leq \varep \|\nabla v\|^2_{L^2} + C(\varep) \|v\|^2_{L^2} = \varep \|\nabla v\|^2_{L^2} + C(\varep) a.
\]
	By the Gagliardo-Nirenberg inequality,
	\[
	\|v\|^{\alpha+2}_{L^{\alpha+2}} \lesssim \|\nabla v\|^{\frac{d\alpha}{2}}_{L^2} \|v\|^{\frac{4-(d-2)\alpha}{2}}_{L^2}.
	\]
	We next apply the Young's inequality with the fact $0<\frac{d\alpha}{2}<2$ to get for any $\varep>0$,
	\begin{align} \label{gag-nir-ine-app}
	\frac{1}{\alpha+2} \|v\|^{\alpha+2}_{L^{\alpha+2}} \leq \frac{\varep}{2} \|\nabla v\|^2_{L^2} + C(\varep, \alpha,a).
	\end{align}
	This shows that for any $\varep>0$, there exists $C(\varep,\alpha,a)>0$ such that
	\begin{align} \label{est-ene}
	E(v) \geq \left(\frac{1}{2}-\varep\right) \|\nabla v\|^2_{L^2} - C(\varep, \alpha,a).
	\end{align}
	By choosing $0<\varep<\frac{1}{2}$, we see that $E(v) \geq -C(\varep, \alpha,a)$. Thus the minimizing problem \eqref{min-prob-a} is well-defined. Let $v^\lambda$ be as in \eqref{scaling}. It is easy to check that $\|v^\lambda\|^2_{L^2}= \|v\|^2_{L^2}= a$ and
	\[
	E(v^\lambda) = \frac{\lambda^2}{2} \|\nabla v\|^2_{L^2} - \frac{\lambda^\sigma}{2} G(v) - \frac{\lambda^\beta}{\alpha+2} \|v\|^{\alpha+2}_{L^{\alpha+2}},
	\]
	where $\beta$ is as in \eqref{def-bet}.	Since $0<\sigma<2$ and $0<\beta<2$, we can find $\lambda_0>0$ small enough so that $E(v^{\lambda_0}) <0$. Taking $C=-E(v^{\lambda_0})>0$, we obtain that $I(a) \leq -C <0$. 
	
	\noindent {\bf Step 2.} We will show that $\mathcal{N}(a) \ne \emptyset$. Let $(v_n)_{n\geq 1}$ be a minimizing sequence for $I(a)$, i.e. $\|v_n\|^2_{L^2}=a$ for all $n\geq 1$ and $E(v_n) \rightarrow I(a)$ as $n\rightarrow \infty$. By the same argument as in the proof of Lemma $\ref{lem-non-emp-M-ome}$, we may assume that $(v_n)_{n\geq 1}$ is a radially symmetric and radially decreasing sequence. Since $E(v_n) \rightarrow I(a)$ as $n\rightarrow \infty$, there exists $C>0$ such that $E(v_n) \leq I(a) +C$ for any $n \geq 1$. By \eqref{est-ene},
	\[
	\left(\frac{1}{2}-\varep\right) \|\nabla v_n\|^2_{L^2} \leq E(v_n) + C(\varep, \alpha,a) \leq I(a) + C(\varep, \alpha,a).
	\]
	Taking $0<\varep<\frac{1}{2}$, we infer that $(v_n)_{n\geq 1}$ is a bounded sequence in $H^1_{\text{rd}}$. Thanks to \eqref{com-emb}, there exist $v\in H^1$ and a subsequence still denoted by $(v_n)_{n\geq 1}$ such that $v_n \rightharpoonup v$ weakly in $H^1$ and $v_n \rightarrow v$ strongly in $L^q$ for any $2<q<\frac{2d}{d-2}$ if $d\geq 3$ ($2<q<\infty$ if $d=1,2$).
	
	Since $G(v_n) \rightarrow G(v)$ as $n\rightarrow \infty$ (see again the proof of Lemma $\ref{lem-non-emp-M-ome}$), we see that $v\ne 0$. In fact, assume by contradiction that $v \equiv 0$. Since $v_n \rightharpoonup 0$ weakly in $H^1$, $v_n \rightarrow 0$ strongly in $L^{\alpha+2}$ and $G(v_n) \rightarrow 0$ as $n\rightarrow \infty$, we learn from Step 1 that
	\[
	0 \leq \liminf_{n\rightarrow \infty} \frac{1}{2}\|\nabla v_n\|^2_{L^2} = \liminf_{n\rightarrow \infty} E(v_n)+\frac{G(v_n)}{2} +\frac{1}{\alpha+2} \|v_n\|^{\alpha+2}_{L^{\alpha+2}} = I(a) \leq -C <0
	\]
	which is a contradiction. We also have that
	\begin{align} \label{est-ene-v-1}
	E(v) \leq \liminf_{n\rightarrow \infty} E(v_n) = I(a).
	\end{align}
	
	We next show that the minimizing problem \eqref{min-prob-a} is attained by $v$. To see this, we write
	\[
	v_n(x) = v(x) + r_n(x),
	\] 
	where $r_n \rightharpoonup 0$ weakly in $H^1$ and $r_n \rightarrow 0$ strongly in $L^q$ with $2<q<\frac{2d}{d-2}$ if $d\geq 3$ ($2<q<\infty$ if $d=1,2$). We have the following expansions:
	\begin{align}
	\|v_n\|^2_{L^2} &= \|v\|^2_{L^2} + \|r_n\|^2_{L^2} + o_n(1), \label{L2-expan} \\
	\|\nabla v_n\|^2_{L^2} &= \|\nabla v\|^2_{L^2} + \|\nabla r_n\|^2_{L^2} + o_n(1), \label{H1-expan} \\
	\|v_n\|^{\alpha+2}_{L^{\alpha+2}} &= \|v\|^{\alpha+2}_{L^{\alpha+2}} + \|r_n\|^{\alpha+2}_{L^{\alpha+2}} + o_n(1), \label{Lalpha-expan} \\
	G(v_n) &= G(v) + G(r_n) + o_n(1), \label{G-expan}
	\end{align}
	as $n\rightarrow \infty$. In particular, we have
	\begin{align} \label{ene-expan}
	E(v_n) = E(v) + E(r_n) + o_n(1)
	\end{align}
	as $n\rightarrow \infty$. The expansions \eqref{L2-expan}, \eqref{H1-expan} and \eqref{Lalpha-expan} are standard. We thus only prove \eqref{G-expan}. To see this, we write
	\[
	G(v_n) = G(v) + G(r_n) + 2 c  \int |x|^{-\sigma} \text{Re }(v(x) \overline{r}_n(x)) dx.
	\]
	We will show that
	\begin{align} \label{G-expan-proof}
	\int |x|^{-\sigma} v(x) \overline{r}_n(x) dx \rightarrow 0
	\end{align}
	as $n\rightarrow \infty$. Without loss of generality, we may assume that $v$ is continuous and compactly supported. 
	
	In the case $0 \notin \text{supp}(v)$, we have
	\[
	\left| \int |x|^{-\sigma} v(x) \overline{r}_n(x) dx \right| \leq \||\cdot|^{-\sigma} v\|_{L^\infty} \int_{\text{supp}(v)} |r_n(x)| dx \rightarrow 0
	\]
	as $n\rightarrow \infty$. Here we have used the fact $r_n\rightharpoonup 0$ weakly in $H^1$ and the compact embedding $H^1 \hookrightarrow L^1_{\text{loc}}$ to show $r_n \rightarrow 0$ strongly in $L^1_{\text{loc}}$. 
	
	In the case $0 \in \text{supp}(v)$, let $\varep>0$. For $\eta>0$ small to be chosen later, we estimate
	\[
	\left| \int |x|^{-\sigma} v(x) \overline{r}_n(x) dx \right| \leq \left[ \int_{B(0,\eta)} + \int_{\text{supp}(v) \backslash B(0,\eta)}\right] |x|^{-\sigma} |v(x)| |r_n(x)| dx =: J_1 + J_2.
	\]
	The term $J_2$ is treated as above, and there exists $n_0 \in \N$ such that for $n\geq n_0$, $J_2 <\frac{\varep}{2}$. For $J_1$, we use the Cauchy-Schwarz inequality and Hardy's inequality to have
	\begin{align*}
	J_1 &\leq \left(\int_{B(0,\eta)} |x|^{-\sigma} |v(x)|^2 dx \right)^{\frac{1}{2}} \left(\int |x|^{-\sigma} |r_n(x)|^2 dx \right)^{\frac{1}{2}} \\
	&\lesssim \left(\int_{B(0,\eta)} |x|^{-\sigma} |v(x)|^2 dx \right)^{\frac{1}{2}} \|r_n\|_{H^1}.
	\end{align*}
	Since $(r_n)_{n\geq 1}$ is bounded in $H^1$ and $v\in H^1$, the dominated convergence implies that for $\eta>0$ small enough, $J_1 \leq \frac{\varep}{2}$. It follows that for $n\geq n_0$, 
	\[
	\left| \int |x|^{-\sigma} v(x) \overline{r}_n(x) dx \right| <\varep.
	\]
	Collecting the above two cases, we prove \eqref{G-expan-proof}. 
	
	We now set $\tilde{v} = \lambda v$ and $\tilde{r}_n = \lambda_n r_n$, where
	\[
	\lambda:= \frac{\sqrt{a}}{\|v\|_{L^2}} \geq 1, \quad \lambda_n := \frac{\sqrt{a}}{\|r_n\|_{L^2}} \geq 1.
	\]
	It is obvious that $\|\tilde{v}\|^2_{L^2} = \|\tilde{r}_n\|^2_{L^2} = a$, hence $E(\tilde{v}) \geq I(a)$ and $E(\tilde{r}_n) \geq I(a)$. We also have that
	\[
	E(\tilde{v}) = \frac{\lambda^2}{2} \|\nabla v\|^2_{L^2} - \frac{\lambda^2}{2} G(v) - \frac{\lambda^{\alpha+2}}{\alpha+2} \|v\|^{\alpha+2}_{L^{\alpha+2}}.
	\]
	This implies that
	\[
	E(v) = \frac{E(\tilde{v})}{\lambda^2} + \frac{\lambda^\alpha-1}{\alpha+2} \|v\|^{\alpha+2}_{L^{\alpha+2}},
	\]
	and
	\[
	E(r_n) = \frac{E(\tilde{r}_n)}{\lambda_n^2} + \frac{\lambda_n^\alpha-1}{\alpha+2} \|r_n\|^{\alpha+2}_{L^{\alpha+2}} \geq \frac{E(\tilde{r}_n)}{\lambda_n^2}.
	\]
	Using \eqref{ene-expan}, we get
	\begin{align*}
	E(v_n) &\geq \frac{E(\tilde{v})}{\lambda^2} + \frac{\lambda^\alpha-1}{\alpha+2} \|v\|^{\alpha+2}_{L^{\alpha+2}} + \frac{E(\tilde{r}_n)}{\lambda_n^2} + o_n(1) \\
	&\geq \frac{\|v\|^2_{L^2}}{a} I(a) +\frac{\lambda^\alpha-1}{\alpha+2} \|v\|^{\alpha+2}_{L^{\alpha+2}} + \frac{\|r_n\|^2_{L^2}}{a} I(a) + o_n(1) \\
	&= \frac{I(a)}{a} \left( \|v\|^2_{L^2} + \|r_n\|^2_{L^2}\right) + \frac{\lambda^\alpha-1}{\alpha+2} \|v\|^{\alpha+2}_{L^{\alpha+2}} + o_n(1).
	\end{align*}
	Taking $n\rightarrow \infty$ and using \eqref{L2-expan} and the fact $v\ne 0$, we  have that
	\[
	I(a) \geq I(a) + \frac{\lambda^\alpha-1}{\alpha+2} \|v\|^{\alpha+2}_{L^{\alpha+2}}.
	\]
	We thus obtain $\lambda \leq 1$, hence $\lambda=1$ hence $\|v\|^2_{L^2} = \|v_n\|^2_{L^2} = a$. This implies that
	\begin{align} \label{est-ene-v-2}
	E(v) \geq I(a).
	\end{align}
	By \eqref{est-ene-v-1} and \eqref{est-ene-v-2}, we obtain $E(v)=I(a)$ and $\|v\|^2=a$ which implies that the minimizing problem \eqref{min-prob-a} is attained by $v$ or $\mathcal{N}(a) \ne \emptyset$.

	Moreover, we also have that $v_n \rightarrow v$ strongly in $H^1$. In fact, by \eqref{L2-expan} and $\|v\|^2_{L^2}= \|v_n\|^2_{L^2}=a$, we get $\|r_n\|_{L^2} \rightarrow 0$ as $n\rightarrow \infty$. Since $r_n \rightharpoonup 0$ weakly in $H^1$, by the uniqueness of the weak limit, $r_n \rightarrow 0$ strongly in $L^2$. By \eqref{ene-expan}, $\lim_{n\rightarrow \infty} E(v_n) = I(a)$, $E(v) = I(a)$, the fact $G(r_n) \rightarrow 0$ and $v_n \rightarrow 0$ strongly in $L^{\alpha+2}$, we see that $\lim_{n\rightarrow \infty} \|\nabla r_n\|^2_{L^2} =0$. This implies that $\lim_{n\rightarrow \infty} \|\nabla v_n\| = \|\nabla v\|_{L^2}$ which together with $v_n \rightharpoonup v$ weakly in $H^1$ imply $v_n \rightarrow v$ strongly in $\dot{H}^1$. Therefore, $v_n \rightarrow v$ strongly in $H^1$. 
	
	\noindent {\bf Step 3.} Let $v \in H^1$ be a complex valued minimizer for $I(a)$. A standard elliptic regularity bootstrap ensures that $v$ is of class $C^1$. By the  diamagnetic inequality, we see that $|v|$ is also a minimizer for $I(a)$. Moreover, by the Euler-Lagrange equation and using the strong maximum principle, we get $|v|>0$ and thus $v \in C^1(\R^d, \C \backslash \{0\})$. Since $E(v) = E(|v|) = I(a)$, it follows that $\|\nabla(|v|)\|^2_{L^2} = \|\nabla v\|^2_{L^2}$. Set $w(x) := \frac{v(x)}{|v(x)|}$. It follows from the fact $|w(x)|^2=1$ for all $x \in \R^d$ that $\text{Re }(\overline{w} (x) \nabla w(x)) =0$,
	\[
	\nabla v(x) = \nabla (|v|)(x) w(x) + |v(x)| \nabla w(x),
	\]
	and thus $|\nabla v(x)|^2 = |\nabla(|v|)(x)|^2 + |v(x)|^2 |\nabla w(x)|^2$ for all $x\in \R^d$. We also have from this and $\|\nabla (|v|)\|^2_{L^2} = \|\nabla v\|_{L^2}^2$ that
	\[
	\int |v(x)|^2 |\nabla w(x)|^2 dx =0
	\]
	which implies $|\nabla w(x)|=0$ for all $x\in \R^d$. Hence $w$ is a constant with $|w|=1$. We infer that there exists $\theta \in \R$ such that $v(x) = e^{i\theta}\phi(x)$, where $\phi(x)= |v(x)|$. We next prove that $\phi$ is radially symmetric and radially decreasing. Let $\phi^*$ be the symmetric rearrangement of $\phi$. It is well-known (see e.g. \cite{Hajaiej}) that
	\[
	\int |x|^{-\sigma} |\phi^*(x)|^2 dx > \int |x|^{-\sigma} |\phi(x)|^2 dx \text{ unless } \phi=\phi^*.
	\]
	By the Polya-Szego's inequality $\|\nabla \phi^*\|_{L^2} \leq \|\nabla \phi\|_{L^2}$ and the fact $\|\phi^*\|_{L^{\alpha+2}} = \|\phi\|_{L^{\alpha+2}}$, it follows that if $\phi \ne \phi^*$, then $E(\phi^*)<E(\phi)$ and $\|\phi^*\|_{L^2}^2=\|\phi\|_{L^2}^2=a$ which contradicts $I(a)= E(v)= E(\phi)$. Therefore, $\phi$ is radially symmetric and radially decreasing.
	
	\noindent {\bf Step 4.} We will show that $\mathcal{N}(a)$ is orbitally stable under the flow of the focusing problem \eqref{NLS-att}. To see this, we argue by contradiction. Note that the existence of global solutions is proved in Theorem $\ref{theo-gwp-1}$. Suppose that there exist sequences $(u_{0,n})_{n\geq 1} \subset H^1$, $(t_n)_{n\geq 1} \subset \R$ and $\varep_0>0$ such that for all $n\geq 1$,
	\begin{align} \label{sta-pro-1}
	\inf_{v\in S_a} \|u_{0,n} - v\|_{H^1} <\frac{1}{n},
	\end{align}
	and
	\begin{align} \label{sta-pro-2}
	\inf_{v\in S_a} \|u_n(t_n) - v\|_{H^1} \geq \varep_0,
	\end{align}
	where $u_n(t)$ is the solution to \eqref{NLS-att} with initial data $u_{0,n}$. By \eqref{sta-pro-1}, we see that for each $n\geq 1$, there exists $v_n \in S_a$ such that
	\begin{align} \label{sta-pro-3}
	\|u_{0,n} - v_n\|_{H^1} <\frac{2}{n}.
	\end{align}
	We thus have a sequence $(v_n)_{n\geq 1} \subset S_a$. By Step 2, there exists $v\in S_a$ such that
	\begin{align} \label{sta-pro-4}
	\lim_{n\rightarrow \infty} \|v_n -v\|_{H^1} =0.
	\end{align}
	By \eqref{sta-pro-3} and \eqref{sta-pro-4}, we have $u_{0,n} \rightarrow v$ in $H^1$ as $n\rightarrow \infty$. It follows that
	\[
	\lim_{n\rightarrow \infty} \|u_{0,n}\|^2_{L^2} = \|v\|^2_{L^2} =a, \quad \lim_{n\rightarrow \infty} E(u_{0,n}) = E(v) = I(a).
	\]
	Thanks to the conservation of mass and energy,
	\[
	\lim_{n\rightarrow \infty} \|u_n(t_n)\|^2_{L^2} = a, \quad \lim_{n\rightarrow \infty} E(u_n(t_n)) = I(a).
	\]
	By the same argument as in Step 2, we prove as well that there exists $\tilde{v} \in S_a$ such that up to a subsequence, $(u_n(t_n))_{n\geq 1}$ converges strongly to $\tilde{v}$ in $H^1$ which contradicts with \eqref{sta-pro-2}. The proof of Theorem $\ref{theo-exi-min-mas-sub}$ is now complete.
	\hfill $\Box$	

	\noindent {\it Proof of Theorem $\ref{theo-exi-gro-mas-cri}$.} The proof is similar to the one of Theorem $\ref{theo-exi-min-mas-sub}$ except \eqref{gag-nir-ine-app} which becomes
	\begin{align} \label{gag-nir-ine-app-mas-cri}
	\frac{1}{\frac{4}{d} + 2} \|v\|^{\frac{4}{d}+2}_{L^{\frac{4}{d} +2}} \leq \frac{1}{2} \left( \frac{\|v\|_{L^2}}{\|Q\|_{L^2}}\right)^{\frac{4}{d}} \|\nabla v\|^2_{L^2}.
	\end{align}
	Thus \eqref{est-ene} is replaced by
	\begin{align*} 
	E(v) \geq \frac{1}{2} \left(1- \left( \frac{\|v\|_{L^2}}{\|Q\|_{L^2}}\right)^{\frac{4}{d}} - \varep \right) \|\nabla v\|^2_{L^2} - C(\varep) a.
	\end{align*}
	Since $\|v\|_{L^2} <\|Q\|_{L^2}$, we choose $0<\varep< 1- \left( \frac{\|v\|_{L^2}}{\|Q\|_{L^2}}\right)^{\frac{4}{d}}$ to get the lower bound for $E(v)$. The rest of the proof follows the same lines as in the proof of Theorem $\ref{theo-exi-min-mas-sub}$. Note that the existence of global solutions is given in Theorem $\ref{theo-gwp-1}$.
	\hfill $\Box$

	\section{Blow-up behavior of standing waves}
	\label{S5}
	In this subsection, we will prove the non-existence of minimizers for $I(a)$ with $a \geq a^*$ in the mass-critical case as well as the blow-up behavior of minimizers for $I(a)$ as $a \nearrow a^*$. 	
	
	\noindent {\it Proof of Theorem $\ref{theo-blo-up-behavior}$.} The proof is done by several steps.
	
	\noindent {\bf Step 1.}	We first show that there is no minimizer for $I(a)$ with $a\geq a^*$. To do this, we pick $\varphi \in C^\infty_0(\R^d)$ satisfying $0 \leq \varphi \leq 1$, $\varphi = 1$ on $|x| \leq 1$ and denote
	\[
	v_\tau(x) := A_\tau \tau^{\frac{d}{2}} \varphi(x) Q_0(\tau x), \quad \tau >0,
	\]
	where $Q_0 = \frac{Q}{\|Q\|_{L^2}}$ and $A_\tau >0$ is such that $\|v_\tau\|^2_{L^2} =a$ for all $\tau>0$. It follows that
	\[
	a A_\tau^{-2} = \int \varphi^2(\tau^{-1} x) Q_0^2(x) dx = 1 + \int (1-\varphi^2(\tau^{-1} x)) Q_0^2(x) dx
	\]
	due to $\|Q_0\|^2_{L^2}=1$. Since $Q_0, |\nabla Q_0| =O(e^{- \delta|x|})$ for some $\delta>0$ as $|x| \rightarrow \infty$, we see that for $\tau$ sufficiently large and $N>0$,
	\[
	\left|\int (1-\varphi^2(\tau^{-1} x)) Q_0^2(x) dx \right| \lesssim \int_{|x|\geq  \tau} e^{-2\delta |x|} dx \lesssim \int_{|x| \geq \tau} |x|^{-d-N} dx \lesssim \tau^{-N}.
	\] 
	This shows that 
	\[
	a A_\tau^{-2} = 1 + O(\tau^{-\infty})
	\]
	as $\tau \rightarrow \infty$. Here $A=O(\tau^{-\infty})$ means that $|A| \leq C \tau^{-N}$ for any $N>0$. We next compute
	\begin{align*}
	\|\nabla v_\tau\|^2_{L^2} = A^2_\tau \left( \int |\nabla \varphi(\tau^{-1} x)|^2 Q_0^2(x) dx \right. &\left. + \ \tau^2 \int \varphi^2(\tau^{-1} x) |\nabla Q_0(x)|^2 dx \right. \\
	&\left. + \ 2\tau \int \text{Re } \left(\varphi(\tau^{-1} x) Q_0(x) \nabla \varphi(\tau^{-1} x) \cdot \nabla Q_0(x) \right) dx \right).
	\end{align*}
	Estimating as above and using the fact $A^2_\tau = a + O(\tau^{-\infty})$ as $\tau \rightarrow \infty$, we get
	\[
	\|\nabla v_\tau\|^2_{L^2} = \tau^2 a \|\nabla Q_0\|^2_{L^2} + O(\tau^{-\infty})
	\]
	as $\tau \rightarrow \infty$. Similarly,
	\begin{align*}
	\int |x|^{-\sigma} |v_\tau(x)|^2 dx & = \tau^\sigma A^2_\tau \int |x|^{-\sigma} \varphi^2(\tau^{-1} x) Q_0^2(x) dx \\
	&= \tau^\sigma a \int |x|^{-\sigma} Q_0^2(x) dx + O(\tau^{-\infty})
	\end{align*}
	and
	\begin{align*}
	\|v_\tau\|^{\frac{4}{d}+2}_{L^{\frac{4}{d}+2}} &=\tau^2 A_\tau^{\frac{4}{d}+2} \int \left(\varphi(\tau^{-1} x) Q_0(x) \right)^{\frac{4}{d}+2} dx \\
	&= \tau^2 a^{\frac{2}{d}+1} \int (Q_0(x))^{\frac{4}{d}+2} dx + O(\tau^{-\infty})
	\end{align*}
	as $\tau \rightarrow \infty$. This implies that
	\begin{align}
	\frac{I(a)}{a} \leq \frac{E(v_\tau)}{a} &= \tau^2 \left( \frac{1}{2} \|\nabla Q_0\|^2_{L^2} - \frac{a^{\frac{2}{d}}}{\frac{4}{d}+2} \|Q_0\|^{\frac{4}{d}+2}_{L^{\frac{4}{d}+2}} \right) - \frac{\tau^\sigma}{2} G(Q_0) + O(\tau^{-\infty}) \nonumber \\
	&= \frac{\tau^2 d}{4} \left( 1- \left( \frac{a}{a^*}\right)^{\frac{2}{d}} \right) - \frac{\tau^\sigma}{2} G(Q_0) + O(\tau^{-\infty})  \nonumber \\
	&= \frac{\tau^2 d}{4} \beta_a - \frac{\tau^\sigma}{2} G(Q_0) + O(\tau^{-\infty}) \label{upperbound-proof}
	\end{align}
	as $\tau \rightarrow \infty$. Here we have used the fact that
	\[
	1= \|Q_0\|^2_{L^2} = \frac{2}{d} \|\nabla Q_0\|^2_{L^2} = \frac{2}{d+2} \|Q\|^{\frac{4}{d}}_{L^2}\|Q_0\|^{\frac{4}{d}+2}_{L^{\frac{4}{d}+2}}
	\]
	which follows from the following Pohozaev's identities
	\[
	\|Q\|^2_{L^2} = \frac{2}{d} \|\nabla Q\|^2_{L^2} = \frac{2}{d+2} \|Q\|^{\frac{4}{d}+2}_{L^{\frac{4}{d}+2}}.
	\]
	We infer from \eqref{upperbound-proof} that for $a\geq a^*$, 
	\[
	\frac{I(a)}{a} \leq \frac{E(v_\tau)}{a} \rightarrow -\infty \text{ as } \tau \rightarrow \infty
	\]
	which shows the non-existence of minimizers for $I(a)$ with $a\geq a^*$. 
	
	\noindent {\bf Step 2.} Let $v_a$ be a non-negative minimizer for $I(a)$ with $0<a<a^*$. We will show that $v_a$ blows up as $a \nearrow a^*$ in the sense of \eqref{blowup}. Assume by contradiction that $(v_a)_{a \nearrow a^*}$ is bounded in $H^1$. We can assume that $v_a$ is radially symmetric and radially decreasing. By the same argument as in the proof of Theorem $\ref{theo-exi-min-mas-sub}$, we show that there exists a minimizer for $I(a^*)$ which is a contradiction.
	
	\noindent {\bf Step 3.} We now claim that there exist two positive constants $m<M$ independent of $a$ such that for $0<a<a^*$,
	\begin{align} \label{energy-estimate}
	-M \beta_a^{-\frac{\sigma}{2-\sigma}} \leq \frac{I(a)}{a} \leq -m \beta_a^{-\frac{\sigma}{2-\sigma}},
	\end{align}
	where $\beta_a$ is as in \eqref{def-bet-a}. To see this, we first show that $\frac{I(a)}{a}$ is a decreasing function in $a$. Indeed, let $0<a \leq b$. We will show that $\frac{I(b)}{b} \leq \frac{I(a)}{a}$. Let $v\in H^1$ be such that $\|v\|^2_{L^2} =a$ and set $\lambda = \sqrt{\frac{b}{a}} \geq 1$. We see that $\|\lambda v\|^2_{L^2} = b$ and
	\begin{align*}
	E(\lambda v) &= \frac{\lambda^2}{2} \|\nabla v\|^2_{L^2} - \frac{\lambda^2}{2} G(v) - \frac{\lambda^{\alpha+2}}{\alpha+2} \|v\|^{\alpha+2}_{L^{\alpha+2}} \\
	&= \lambda^2 E(v) +\frac{\lambda^2(1-\lambda^\alpha)}{\alpha+2} \|v\|^{\alpha+2}_{L^{\alpha+2}} \leq \lambda^2 E(v).
	\end{align*}
	We then have from the definition of $I(b)$ that
	\[
	I(b) \leq E(\lambda v) \leq \lambda^2 E(v).
	\]
	Taking the infimum over all $v \in H^1$ with $\|v\|^2_{L^2} = a$, we obtain $I(b) \leq \frac{b}{a} I(a)$ which shows that $\frac{I(a)}{a}$ is a decreasing function in $a$. Thus, we only need to show \eqref{energy-estimate} for $a$ close to $a^*$. 
	
	We now have from Hardy's inequality and Young's inequality with $0<\sigma<2$ that for any $\varep>0$, 
	\begin{align} \label{young}
	G(v) \leq C \|\nabla v\|^\sigma_{L^2} \|v\|^{2-\sigma}_{L^2} \leq \varep \|\nabla v\|^2_{L^2} + C(\sigma) \varep^{-\frac{\sigma}{2-\sigma}} \|v\|^2_{L^2}.
	\end{align}
	Note that the constant $C$ may change from line to line. The above estimate together with the sharp Gagliardo-Nirenberg inequality \eqref{gag-nir-ine-app-mas-cri} imply that
	\[
	E(v) \geq \frac{1}{2}\left(1- \left(\frac{\|v\|_{L^2}}{\|Q\|_{L^2}}\right)^{\frac{4}{d}} -\varep \right) \|\nabla v\|^2_{L^2} - \frac{C(\sigma)}{2} \varep^{-\frac{\sigma}{2-\sigma}} \|v\|^2_{L^2}.
	\]
	Let $v \in H^1$ be such that $\|v\|^2_{L^2}=a$. It follows that
	\begin{align*}
	E(v) &\geq \frac{1}{2} \left( 1- \left(\frac{a}{a^*}\right)^{\frac{2}{d}} -\varep \right) \|\nabla v\|^2_{L^2} - \frac{C(\sigma)}{2} \varep^{-\frac{\sigma}{2-\sigma}} a \\
	&=\frac{1}{2} \left( \beta_a -\varep \right) \|\nabla v\|^2_{L^2} - \frac{C(\sigma)}{2} \varep^{-\frac{\sigma}{2-\sigma}} a,
	\end{align*}
	where $\beta_a$ is as in \eqref{def-bet-a}. We take $\varep= \frac{1}{2} \beta_a$ and get
	\[
	E(v) \geq - 2^{\frac{-2+2\sigma}{2-\sigma}} C(\sigma) \beta_a^{-\frac{\sigma}{2-\sigma}} a.
	\]
	Taking the infimum over all $v\in H^1$ with $\|v\|^2_{L^2} =a$, we prove the lower bound in \eqref{energy-estimate} with $M= 2^{\frac{-2+2\sigma}{2-\sigma}} C(\sigma)$. To see the upper bound in \eqref{energy-estimate}, we choose $\tau = \lambda \beta_a^{-\frac{1}{2-\sigma}}$ in \eqref{upperbound-proof} with $\lambda>0$. Note that $\tau \rightarrow \infty$ as $a \nearrow a^*$ since $\beta_a \rightarrow 0$ as $a \nearrow a^*$. With this choice, \eqref{upperbound-proof} becomes
	\begin{align} \label{estimate-Ia}
	\frac{I(a)}{a} \leq \left(\frac{\lambda^2 d}{4} - \frac{\lambda^\sigma}{2} G(Q_0) \right) \beta_a^{-\frac{\sigma}{2-\sigma}} + o_{a\nearrow a^*}(1)
	\end{align}
	for any $\lambda>0$. Since $0<\sigma<2$, there exists $\lambda_0 >0$ sufficiently small so that $-2m:=\frac{\lambda_0^2 d}{4} - \frac{\lambda_0^\sigma}{2} G(Q_0) <0$ and $m<M$. Taking $a$ sufficiently close to $a^*$, we prove the upper bound in \eqref{energy-estimate}.
	We also have from \eqref{estimate-Ia} that
	\begin{align} \label{limsup-Ia}
	\limsup_{a\nearrow a^*} \beta_a^{\frac{\sigma}{2-\sigma}} \frac{I(a)}{a} \leq \inf_{\lambda>0} \left( \frac{\lambda^2 d}{4} - \frac{\lambda^\sigma}{2} G(Q_0) \right).
	\end{align}
	
	{\bf Step 4.} Let $v_a$ be a non-negative minimizer for $I(a)$ with $0<a<a^*$. We claim that then there exists $K>1$ independent of $a$ such that for $0<a<a^*$,
	\begin{align} \label{potential-estimate}
	-K \beta_a^{-\frac{\sigma}{2-\sigma}} \leq - \frac{G(v_a)}{a} \leq -\frac{1}{K} \beta_a^{-\frac{\sigma}{2-\sigma}}
	\end{align}
	and
	\begin{align} \label{kinetic-estimate}
	\frac{\|\nabla v_a\|^2_{L^2}}{a} \leq K \beta_a^{-\frac{2}{2-\sigma}}.
	\end{align}
	The upper bound in \eqref{potential-estimate} follows easily from the upper bound in \eqref{energy-estimate} and the fact
	\[
	I(a)=E(v_a) \geq \frac{1}{2} \left(1- \left(\frac{a}{a^*} \right)^{\frac{2}{d}} \right) \|\nabla v_a\|^2_{L^2} - \frac{G(v_a)}{2} \geq - \frac{G(v_a)}{2}.
	\]
	To see the lower bound in \eqref{potential-estimate}, we use again \eqref{young} and the fact $E(v_a)= I(a)<0$ to have that
	\begin{align*}
	-\frac{G(v_a)}{2} \geq E(v_a) - \frac{G(v_a)}{2} &\geq \frac{1}{2} \left( 1- \left(\frac{a}{a^*}\right)^{\frac{2}{d}} - 2\varep\right) \|\nabla v_a\|^2_{L^2} - C(\sigma) \varep^{-\frac{\sigma}{2-\sigma}} a \\
	&= \frac{1}{2} (\beta_a - 2\varep) \|\nabla v_a\|^2_{L^2} - C(\sigma) \varep^{-\frac{\sigma}{2-\sigma}} a.
	\end{align*}
	We take $\varep = \frac{1}{4}\beta_a$ and get 
	\[
	-\frac{G(v_a)}{a} \geq \frac{1}{4} \beta_a \|\nabla v_a\|^2_{L^2}-2^{\frac{2+\sigma}{2-\sigma}} C(\sigma) \beta_a^{-\frac{\sigma}{2-\sigma}} \geq -2^{\frac{2+\sigma}{2-\sigma}} C(\sigma) \beta_a^{-\frac{\sigma}{2-\sigma}}.
	\]
	The above estimate also gives \eqref{kinetic-estimate}.
	
	\noindent {\bf Step 5.} We finally show the blow-up behavior of minimizers for $I(a)$ as $a \nearrow a^*$. To this end, we denote
	\begin{align} \label{define-wa}
	w_a(x) := \beta_a^{\frac{d}{2(2-\sigma)}} v_a \Big(\beta_a^{\frac{1}{2-\sigma}} x\Big).
	\end{align}
	It follows that
	\[
	\|w_a\|^2_{L^2} = \|v_a\|^2_{L^2} = a
	\]
	and
	\[
	\|\nabla w_a\|^2_{L^2} = \beta_a^{\frac{2}{2-\sigma}} \|\nabla v_a\|^2_{L^2} \leq K a
	\]
	by \eqref{kinetic-estimate}. This implies that $(w_a)_{a \nearrow a^*}$ is a bounded sequence in $H^1$. Up to a subsequence, $w_a \rightharpoonup w$ weakly in $H^1$ and pointwise almost everywhere. By the lower continuity of the weak limit, 
	\begin{align} \label{mass-w}
	\|w\|^2_{L^2} \leq \liminf_{a \nearrow a^*} \|w_a\|^2_{L^2} = a^*.
	\end{align}
	By \eqref{potential-estimate} and the weak continuity of the potential energy (see e.g. \cite[Theorem 11.4]{LL}),
	\[
	\frac{1}{K} \leq \frac{\beta_a^{\frac{\sigma}{2-\sigma}} G(v_a)}{a} = \frac{G(w_a)}{a} \rightarrow \frac{G(w)}{a^*}
	\]
	as $a \nearrow a^*$. This shows that $w \ne 0$. 
	
	We next show that $w$ is actually a non-negative optimizer for the sharp Gagliardo-Nirenberg inequality \eqref{sharp-GN-inequality}. In fact, we have from \eqref{energy-estimate}, \eqref{potential-estimate} and the fact $0<\sigma <2$ that
	\begin{align} \label{convergence-wa}
	0 &= \lim_{a \nearrow a^*} \beta_a^{\frac{2}{2-\sigma}} \left( \frac{E(v_a)}{a} + \frac{G(v_a)}{2a} \right) \nonumber \\
	&= \lim_{a\nearrow a^*} \frac{\beta_a^{\frac{2}{2-\sigma}}}{a} \left( \frac{1}{2} \|\nabla v_a\|^2_{L^2} - \frac{1}{\frac{4}{d}+2} \|v_a\|^{\frac{4}{d}+2}_{L^{\frac{4}{d}+2}} \right) \nonumber \\
	&= \lim_{a \nearrow a^*} \frac{1}{a} \left(\frac{1}{2} \|\nabla w_a\|^2_{L^2} - \frac{1}{\frac{4}{d}+2} \|w_a\|^{\frac{4}{d}+2}_{L^{\frac{4}{d}+2}}  \right).
	\end{align}
	Since $w_a \rightharpoonup w$ weakly in $H^1$, we have 
	\begin{align}
	\|w_a\|^2_{L^2} &= \|w\|^2_{L^2} + \|w_a - w\|^2_{L^2} + o_{a\nearrow a^*} (1), \label{mass-expansion}\\
	\|\nabla w_a\|^2_{L^2} &= \|\nabla w\|^2_{L^2} + \|\nabla (w_a - w)\|^2_{L^2} + o_{a\nearrow a^*} (1). \label{kinetic-expansion}
	\end{align}
	Since $w_a \rightarrow w$ pointwise almost everywhere and $(w_a)_{a \nearrow a^*}$ is bounded in $H^1$, the Brezis-Lieb's lemma (see e.g. \cite{BL}) implies that
	\begin{align} \label{nonlinear-expansion}
	\|w_a\|^{\frac{4}{d}+2}_{L^{\frac{4}{d}+2}} = \|w\|^{\frac{4}{d}+2}_{L^{\frac{4}{d}+2}} + \|w_a-w\|^{\frac{4}{d}+2}_{L^{\frac{4}{d}+2}} + o_{a \nearrow a^*} (1).
	\end{align}
	It follows from \eqref{convergence-wa}, \eqref{kinetic-estimate} and \eqref{nonlinear-expansion} that
	\begin{align} \label{convergence-wa-w}
	\frac{1}{2} \|\nabla w\|^2_{L^2} - \frac{1}{\frac{4}{d}+2} \|w\|^{\frac{4}{d}+2}_{L^{\frac{4}{d}+2}} + \frac{1}{2} \|\nabla (w_a-w)\|^2_{L^2} - \frac{1}{\frac{4}{d}+2} \|w_a-w\|^{\frac{4}{d}+2}_{L^{\frac{4}{d}+2}} = o_{a\nearrow a^*}(1).
	\end{align}
	Using the sharp Gagliardo-Nirenberg inequality and \eqref{mass-w}, we see that
	\[
	\frac{1}{2} \|\nabla w\|^2_{L^2} - \frac{1}{\frac{4}{d}+2} \|w\|^{\frac{4}{d}+2}_{L^{\frac{4}{d}+2}} \geq \frac{1}{2} \left(1-\left( \frac{\|w\|^2_{L^2}}{a^*}\right)^{\frac{2}{d}}\right) \|\nabla w\|^2_{L^2} \geq 0
	\]
	and similarly,
	\[
	\frac{1}{2} \|\nabla (w_a-w)\|^2_{L^2} - \frac{1}{\frac{4}{d}+2} \|w_a-w\|^{\frac{4}{d}+2}_{L^{\frac{4}{d}+2}} \geq \frac{1}{2} \left(1-\left( \frac{\|w_a-w\|^2_{L^2}}{a^*}\right)^{\frac{2}{d}}\right)  \|\nabla (w_a -w)\|^2_{L^2} \geq 0.
	\]
	We infer from the above inequalities, \eqref{convergence-wa-w} and the fact $w \ne 0$ that
	\[
	\frac{1}{2} \|\nabla w\|^2_{L^2} - \frac{1}{\frac{4}{d}+2} \|w\|^{\frac{4}{d}+2}_{L^{\frac{4}{d}+2}}=0, \quad \|w\|^2_{L^2}=a^*, \quad \lim_{a \nearrow a^*} \|\nabla(w_a -w)\|^2_{L^2}=0.
	\]
	This shows that $w$ is a non-negative optimizer for the sharp Gagliardo-Nirenberg inequality \eqref{sharp-GN-inequality}. Moreover, the later limit together with $w_a \rightharpoonup w$ weakly in $H^1$ imply that $w_a \rightarrow w$ strongly in $H^1$. Since $Q$ is the unique optimizer for the sharp Gagliardo-Nirenberg inequality up to translations and dialations (see e.g. \cite{Weinstein}), we conclude that $w(x) = \beta_0 Q(\lambda_0 x - x_0)$ for some $\beta_0, \gamma_0>0$ and $x_0 \in \R^d$. Since $\|w\|^2_{L^2} = a^* = \|Q\|^2_{L^2}$, we infer that $\beta_0 = \lambda_0^{\frac{d}{2}}$, hence
	\[
	w(x) = \lambda_0^{\frac{d}{2}} Q(\lambda_0 x - x_0)
	\]
	for some $\lambda_0>0$ and $x_0 \in \R^d$. It remains to determine $\lambda_0$ and $x_0$ as follows. 
	
	We have from \eqref{define-wa} that
	\[
	G(v_a) = \beta_a^{-\frac{\sigma}{2-\sigma}} G(w_a)
	\]
	and by the sharp Gagliardo-Nirenberg inequality,
	\begin{align*}
	\frac{1}{2} \|\nabla v_a\|^2_{L^2} - \frac{1}{\frac{4}{d}+2} \|v_a\|^{\frac{4}{d}+2}_{L^{\frac{4}{d}+2}} &\geq \frac{1}{2} \left( 1-\left(\frac{a}{a^*}\right)^{\frac{2}{d}} \right) \|\nabla v_a\|^2_{L^2}  \\
	&= \frac{1}{2} \beta_a \|\nabla v_a\|^2_{L^2} \\
	&= \frac{1}{2} \beta_a^{-\frac{\sigma}{2-\sigma}} \|\nabla w_a\|^2_{L^2}.
	\end{align*}
	It follows that
	\[
	\frac{I(a)}{a} = \frac{E(v_a)}{a} \geq \frac{\beta_a^{-\frac{\sigma}{2-\sigma}} }{a} \left( \frac{1}{2} \|\nabla w_a\|^2_{L^2} - \frac{1}{2} G(w_a) \right).
	\]
	Since $w_a \rightarrow w$ strongly in $H^1$ and $G(w_a) \rightarrow G(w)$ as $a \nearrow a^*$, we see that
	\begin{align*}
	\beta_a^{\frac{\sigma}{2-\sigma}} \frac{I(a)}{a} &\geq \frac{1}{2} \left( \frac{\|\nabla w\|^2_{L^2}}{a^*} - \frac{G(w)}{a^*} \right) + o_{a\nearrow a^*} (1) \\
	&= \frac{1}{2} \left( \frac{\lambda_0^2 \|\nabla Q\|^2_{L^2}}{a^*} - \frac{\lambda_0^\sigma G(Q(\cdot-x_0))}{a^*} \right) + o_{a\nearrow a^*} (1) \\
	&= \frac{\lambda_0^2 d}{4} - \frac{\lambda_0^\sigma}{2} G(Q_0(\cdot-x_0)) + o_{a \nearrow a^*} (1).
	\end{align*}
	This implies that
	\begin{align} \label{liminf-Ia}
	\liminf_{a\nearrow a^*} \beta_a^{\frac{\sigma}{2-\sigma}} \frac{I(a)}{a} \geq \frac{\lambda_0^2 d}{4} - \frac{\lambda_0^\sigma}{2} G(Q_0(\cdot-x_0)) \geq \frac{\lambda_0^2 d}{4} - \frac{\lambda_0^\sigma}{2} G(Q_0),
	\end{align}
	where the last inequality follows from the Hardy-Littlewood rearrangement inequality and the fact $Q_0$ is radially symmetric and radially decreasing. Note that the equality holds if and only if $x_0 = 0$. We infer from \eqref{limsup-Ia} and \eqref{liminf-Ia} that $x_0 =0$ and
	\[
	\lim_{a \nearrow a^*} \beta_a^{\frac{\sigma}{2-\sigma}} \frac{I(a)}{a} = \frac{\lambda_0^2 d}{4} - \frac{\lambda_0^\sigma}{2} G(Q_0) = \min_{\lambda>0} \left(\frac{\lambda^2 d}{4} - \frac{\lambda^\sigma}{2} G(Q_0) \right).
	\]
	A direct computation shows that
	\[
	\lambda_0 = \left( \frac{\sigma G(Q_0)}{d} \right)^{\frac{1}{2-\sigma}}. 
	\]
	In conclusion, we have proved that up to a subsequence,
	\[
	\beta_a^{\frac{d}{2(2-\sigma)}} v_a(\beta_a^{\frac{1}{2-\sigma}} \cdot) \rightarrow \lambda_0^{\frac{d}{2}} Q(\lambda_0 \cdot) \text{ strongly in } H^1  \text{ as } a \nearrow a^*
	\]
	with $\lambda_0 = \left( \frac{\sigma G(Q_0)}{d} \right)^{\frac{1}{2-\sigma}}$. Moreover, by the uniqueness of $Q$, we can conclude the above limit holds for the whole sequence $(v_a)_{a \nearrow a^*}$. 
	
	\section*{Appendix}
	
	\paragraph{\bf The radial compact embedding}
	We first give a proof of the radial compact embedding \eqref{com-emb}. It is enough to show that if $(v_n)_{n\geq 1} \subset H^1_{\text{rd}}$ is such that $v_n\rightharpoonup 0$, then $v_n\rightarrow 0$ in $L^q$. Of course we can assume that $v_n \rightarrow 0$ in $L^q_{\text{loc}}$ and $v_n \rightarrow 0$ almost everywhere. It follows that
	\[
	\|v_n\|^q_{L^q} = \|v_n\|^q_{L^q(B)} + \|v_n\|^q_{L^q(B^c)} = \|v_n\|^q_{L^q(B^c)} + o_n(1)
	\]
	as $n\rightarrow \infty$, where $B=\{ x \in \R^d \ : \ |x| <1\}$ and $B^c= \R^d \backslash B$. Since $(v_n)_{n\geq 1}$ is bounded in $H^1$, by \eqref{rad-decrea-est}, there exists $C(d)>0$ independent of $n$ such that
	\[
	|v_n(x)|^q \leq C(d) |x|^{-\frac{dq}{2}}.
	\]
	The last term is integrable on $B^c$ since $\frac{dq}{2}>d$. By the dominated convergence,
	\[
	\|v_n\|^q_{L^q(B^c)} \rightarrow 0
	\]
	as $n\rightarrow \infty$. 
	
	\paragraph{\bf The uniqueness of positive radial solution for \eqref{ordinary-equation}}
	Using the general results of Shioji-Watanabe \cite{SW}, we have that for any $\omega>-\mu_1$, $0<\sigma<1$, $d\geq 3$ and $0<\alpha<\frac{4}{d-2}$, there exists a unique positive solution to \eqref{ordinary-equation}. In fact, in \cite[Theorem 1]{SW}, Shioji-Watanabe proved a uniqueness result for 
	\[
	\phi''(r) + \frac{f'(r)}{f(r)} \phi'(r) + g(r) \phi(r) + h(r) \phi^p(r) =0, \quad r \in (0,+\infty)
	\]
	under appropriate assumptions on $f(r), g(r)$ and $h(r)$. In our case, we have
	\[
	f(r) = r^{d-1}, \quad g(r) = cr^{-\sigma}-\omega, \quad h(r)=1, \quad p=\alpha+1.
	\]
	Applying Theorem 1 in \cite{SW}, the uniqueness of positive solution to \eqref{ordinary-equation} holds if the following conditions are satisfied:
	\begin{itemize}
		\item[(I)] $g \in C^1(0,+\infty)$, $h \in C^3(0,+\infty)$ and $h(r)>0$ for every $r \in(0,+\infty)$.
		\item[(II)] $$\lim_{r\rightarrow 0} r^{1-d} \int_0^r \tau^{d-1} (|g(\tau)| + h(\tau)) d\tau =0.$$
		\item[(III)] There exists $r_0 \in (0,+\infty)$ such that 
		\begin{itemize}
			\item[(i)] $r^{d-1} g(r), r^{d-1} h(r) \in L^1(0,r_0)$;
			\item[(ii)] $$r^{d-1} (|g(r)| +h(r)) \left(\frac{r_0^{2-d} - r^{2-d}}{2-d} \right) \in L^1(0,r_0).$$
		\end{itemize}
		\item[(IV)] $\lim_{r\rightarrow 0} a(r) <\infty$, $\lim_{r\rightarrow 0} |b(r)|<\infty$, $\lim_{r\rightarrow 0} c(r) \in [0,\infty]$, $\lim_{r\rightarrow 0} a(r) g(r) =0$ and $\lim_{r\rightarrow 0} a(r) h(r) =0$, where
		\begin{align*}
		a(r) &= r^{\frac{2(d-1)(p+1)}{p+3}} h(r)^{-\frac{2}{p+3}}, \\
		b(r) &=-\frac{1}{2} a'(r) + \frac{d-1}{r} a(r), \\
		c(r) &= -b'(r) +\frac{d-1}{r} b(r).
		\end{align*}
		\item[(V)] There exists $r_1 \in (0,\infty)$ such that $G(r) >0$ on $(0,r_1)$ and $G(r)<0$ on $(r_1,+\infty)$, where
		\[
		G(r) = -b(r) g(r) + \frac{1}{2} c'(r) + \frac{1}{2} (a'(r) g(r) + a(r) g'(r)).
		\]
		\item[(VI)] $G^- \neq 0$, where $G^- := \min \{G(r),0\}$ for $r \in (0,+\infty)$.
	\end{itemize}
	
Under the assumptions $\omega>-\mu_1$, $0<\sigma<1$, $d\geq 3$ and $0<\alpha<\frac{4}{d-2}$, it is easy to see that (I), (II) and (III) hold. By a direct computation, we have
\begin{align*}
a(r) &= r^{\frac{2(d-1)(\alpha+2)}{\alpha+4}}, \\
b(r) &= \frac{2(d-1)}{\alpha+4} r^{\frac{(2d-3)\alpha + 4(d-2)}{\alpha+4}},\\
c(r) &= \frac{2(d-1) (4-(d-2)\alpha)}{(\alpha+4)^2} r^{\frac{2(d-2)\alpha + 4(d-3)}{\alpha+4}}.
\end{align*}
Since $d\geq 3$ and $0<\sigma<1$, we see that (IV) is satisfied. Finally, we have
\begin{align*}
G(r) &= A r^{\frac{(2d-3)\alpha +4(d-2)}{\alpha+4}} + B r^{\frac{(2d-3)\alpha +4(d-2)}{\alpha+4} - \sigma} + C r^{\frac{(2d-3)\alpha +4(d-2)}{\alpha+4}-2} \\
&= (Ar^2+ Br^{2-\sigma} + C) r^{\frac{(2d-3)\alpha +4(d-2)}{\alpha+4}-2},
\end{align*}
where
\begin{align*}
A&= -\frac{\omega \alpha(d-1)}{\alpha+4}, \\
B&= \frac{c((2d-2-\sigma)\alpha - 4\sigma)}{2(\alpha+4)}, \\
C&= \frac{(d-1) (4-(d-2)\alpha) [2(d-2)\alpha+4(d-3)]}{(\alpha+4)^3}.
\end{align*}
Since $A<0$ and $C>0$, there exists $r_1 >0$ such that $G(r)>0$ on $(0,r_1)$ and $G(r)<0$ on $(r_1,\infty)$ which shows (V) and aslo (VI). 
	
	\section*{Acknowledgement}
	This work was supported in part by the Labex CEMPI (ANR-11-LABX-0007-01). The author would like to express his deep gratitude to his wife - Uyen Cong for her encouragement and support. He also would like to thank the reviewer for his/her helpful comments and suggestions.

\end{document}